\documentclass[10pt,twoside,a4paper]{amsart}

\usepackage[left=3.5cm,right=3.5cm,top=3.5cm,bottom=3cm]{geometry}
\usepackage[T1]{fontenc}
\usepackage{amsmath}
\usepackage{combelow}
\usepackage{amssymb}
\usepackage{amsthm}
\usepackage{thmtools}
\usepackage{enumitem}
\usepackage{mathrsfs}
\usepackage{verbatim}
\usepackage{stmaryrd}
\usepackage{bm}
\usepackage{bbm}
\usepackage{tikz}
\usepackage[colorlinks=true, linktocpage=true, linkcolor=red!70!black, citecolor=green!50!black]{hyperref}

\usepackage{todonotes}

\setenumerate{label={\rm (\alph{*})}}

\newcommand{\FF}{{\bm F}}
\newcommand{\GG}{{\bm G}}

\newcommand{\tin}{\mathrm{in}}
\newcommand{\tout}{\mathrm{out}}

\providecommand{\eps}{\varepsilon}

\let\emptyset\varnothing

\newcommand{\weaksto}{\overset{*}{\rightharpoonup}}
\newcommand{\weakstarto}{\overset{*}{\rightharpoonup}}
\renewcommand{\d}{\mathrm{d}}
\newcommand{\D}{\mathrm{D}}
\AtBeginDocument{
  \let\div\relax
  \DeclareMathOperator{\div}{div}
}
\DeclareMathOperator{\spt}{spt}
\DeclareMathOperator{\dist}{dist}

\DeclareMathOperator{\ext}{ext}
\DeclareMathOperator{\tr}{tr}

\DeclareMathOperator{\Id}{Id}
\DeclareMathOperator{\Lip}{Lip}
\DeclareMathOperator{\loc}{loc}

\DeclareMathOperator{\sep}{sep}

\DeclareMathOperator*{\wslim}{w^{\ast}-lim}
\newcommand{\mres}{\mathbin{\vrule height 1.6ex depth 0pt width 0.13ex\vrule height 0.13ex depth 0pt width 1.3ex}}
\DeclareMathOperator{\LL}{L}
\DeclareMathOperator{\CC}{C}
\DeclareMathOperator{\WW}{W}

\newcommand{\rmc}{\mathrm{c}}
\newcommand{\rmb}{\mathrm{b}}

\theoremstyle{definition}
\newtheorem{thm}{Theorem}[section]
\newtheorem*{thm*}{Theorem}
\newtheorem{cor}[thm]{Corollary}
\newtheorem{defn}[thm]{Definition}
\newtheorem{definition}[thm]{Definition}
\newtheorem{eg}[thm]{Example}
\newtheorem{lem}[thm]{Lemma}

\newtheorem{prop}[thm]{Proposition}
\newtheorem{rem}[thm]{Remark}
\newtheorem{remark}[thm]{Remark}

\makeatletter
 \newcommand{\linkdest}[1]{\Hy@raisedlink{\hypertarget{#1}{}}}
\makeatother

\usepackage[
  backend=biber,
  sorting=nyt,
  date=year,
  sortcites,
  style=alphabetic,
  giveninits=true,
  maxbibnames=10,
  maxcitenames=10,
  url=false,
  doi=true
]{biblatex}

\bibliography{anzelotti_pairing}

\numberwithin{equation}{section}

\begin{document}

\title[Normal trace space of $\mathcal{DM}$-fields]{On the normal trace space of \\extended divergence-measure fields}
\author{Christopher Irving}
\address{Department of Mathematics and Statistics,  Georgetown University, 3700 Reservoir Road NW, Washington, D.C., 20057, USA}
\email{ci152@georgetown.edu}
\date{\today}
\keywords{Divergence-measure fields, normal trace, Gauss-Green formula, Arens-Eells space}
\subjclass[2020]{35R06, 49Q15}

\maketitle

\begin{abstract}
  We characterise the normal trace space associated to extended (measure-valued) divergence-measure fields on the boundary of a set $E \subset \mathbb R^n$, as the Arens-Eells space $\text{\AE}(\partial E)$.
  Such a trace operator is constructed for any Borel set $E$, and under a mild regularity condition, which includes Lipschitz domains, this trace operator is shown to moreover be surjective.
  This relies in part on a new pointwise description of the Anzellotti pairing $\overline{\nabla \phi \cdot {\boldsymbol F}}$ between a $\mathrm{W}^{1,\infty}$ function $\phi$ and extended divergence-measure field ${\boldsymbol F}$.
  As an application, we prove extension theorems for divergence-measure fields and divergence-free measures.
  Results for $\mathrm{L}^1$-fields are also obtained.
 \end{abstract}

\section{Introduction}

We are concerned with the structure of the distributional normal trace $\langle \FF \cdot \nu, \,\cdot\,\rangle_{\partial E}$ associated to (extended) divergence-measure fields.
Motivated by problems from continuum mechanics, we seek to understand whether the classical Gauss-Green formula,
\begin{equation}\label{eq:classical_gauss_green}
  \int_{\partial U} \phi \,\FF \cdot \nu_{\partial U} \,\d\mathcal H^{n-1} = -\int_U \nabla\phi \cdot \FF \,\d x - \int_U \phi \,\div \FF \,\d x,
\end{equation} 
valid for suitably regular domains $U \subset \mathbb R^n$ with inwards pointing normal $\nu_{\partial U}$, vector fields $\FF$ and scalar functions $\phi$, remains valid when the vector field is highly irregular.
We will consider \emph{divergence-measure fields} $\FF \in \mathcal{DM}^{\ext}(\Omega)$; these are vector-valued Radon measures $\FF$ defined on some open set $\Omega \subset \mathbb R^n$, whose divergence is also a measure on $\Omega$.
Given such a field $\FF$ and any Borel subset $E \subset \Omega$, it is by now customary (following \emph{e.g.}\,\cite{AmbrosioEtAl2005,ChenFrid2003,Silhavy2009,Schuricht2007}) to define the normal trace as a distribution \emph{via} the Gauss-Green formula as
\begin{equation}\label{eq:normal_trace}
\langle \FF \cdot \nu, \phi \rangle_{\partial E} := -\int_E \nabla \phi \cdot \d \FF - \int_E \phi\,\d(\div \FF) \quad\mbox{for all $\phi \in \CC^1_{\rmb}(\Omega)$},
\end{equation} 
where the minus sign corresponds to an inwards pointing normal.
We thereby recast the problem of understanding the Gauss-Green formula to that of studying this distributional normal trace.

The case where the underlying field is represented by a function in $\LL^{\infty}$ has been extensively studied, starting with the seminal papers of \textsc{Anzellotti} \cite{Anzellotti1983} and \textsc{Chen \& Frid} \cite{ChenFrid1999}.
Here the normal trace as defined in \eqref{eq:normal_trace} can be represented by a function in $\LL^{\infty}(\partial U,\mathcal H^{n-1})$, which holds in the generality of sets of finite perimeter, as was established independently by \textsc{Chen \& Torres} in \cite{ChenTorres2005} and \textsc{\v{S}ilhav\'y} in \cite{Silhavy2005}.
While the bounded case will not be the focus of this paper, we will mention there have been many interesting developments since, such as \cite{ChenTorresZiemer2009,ChenLiTorres2020,CrastaDeCicco2019,ComiEtAl2024,ComiLeonardi2025,PhucTorres2008,SchevenSchmidt2016}.

Comparatively little is known when underlying field $\FF$ is unbounded or measure-valued. 
A major difficultly stems from the fact that, even for regular domains $U$, the normal trace $\langle \FF \cdot \nu,\,\cdot\,\rangle_{\partial U}$ need not be represented by a locally $\mathcal{H}^{n-1}$-integrable function on the boundary (\emph{c.f.}\,\cite{ChenFrid2003,Schuricht2007}). 
To make matters worse, the normal trace may fail to even be a measure, as was shown in \cite{Silhavy2009}.
For this reason, many authors study the validity of the Gauss-Green formula on \emph{almost all} surfaces in a suitable sense, starting with \cite{Silhavy1991,DegiovanniEtAl1999} in the context of Cauchy fluxes, which was further developed in \cite{ChenFrid2003,ChenEtAl2019,ChenEtAl2024,Schuricht2007,Silhavy2005,Silhavy2009}.
We also mention \cite{ComiEtAl2023,SchonherrSchuricht2025,Silhavy2008} for further contributions in this setting.

A natural question is whether we can \emph{characterise} the space of distributions that arise as the normal trace of a divergence-measure field.
In this direction \textsc{\v{S}ilhav\'y} in \cite{Silhavy2009} proved that, for any open set $U \subset \mathbb R^n$, we can view the normal trace as a linear functional
\begin{equation}\label{eq:intro_lipdual}
  \mathrm{N}_{U}(\FF) \in \Lip_{\rmb}(\partial U)^{\ast}.
\end{equation} 
As far as the author is aware, in this generality, this is the sharpest description available in the literature up until now.
The purpose of the present work is to settle this problem by showing that \eqref{eq:intro_lipdual} is \emph{not} optimal, by identifying the correct space that the normal trace surjects onto.
As a consequence of this characterisation, we will also obtain extension results for extended divergence-measure fields; these appear to be entirely new.

\subsection{Main results}

We will first establish the following refinement of \eqref{eq:intro_lipdual}, where we show that the normal trace enjoys improved continuity properties.

\begin{thm}\label{thm:intro_normaltrace}
  Let $\FF \in \mathcal{DM}^{\ext}(\Omega)$ and $E \subset \Omega$ be any Borel set.
  Then there exists a unique weakly${}^{\ast}$-continuous linear functional $\mathrm N_E(\FF) \in \Lip_{\rmb}(\partial E)^{\ast}$ satisfying
  \begin{equation}
    \langle\mathrm N_E(\FF),\phi\rvert_{\partial E}\rangle = - \int_E \nabla\phi\cdot\d \FF - \int_E \phi\,\d(\div \FF) \quad\mbox{for all $\phi \in \CC^1_{\rmb}(\Omega)$}.
  \end{equation} 
\end{thm}

A more precise statement is given in Theorem \ref{lem:optimal_trace} which additionally asserts that, as a consequence of the weak${}^{\ast}$ continuity, $\mathrm{N}_U(\FF)$ in fact lies in the \emph{predual} of $\Lip_{\rmb}(\partial E)$, known as the \emph{Arens-Eells space} $\text{\AE}(\partial E)$ (see Definition~\ref{defn:arenseells}).
Theorem~\ref{thm:intro_normaltrace} will rely on a fine description of the Anzellotti-type \emph{pairing measure} $\overline{\nabla \phi \cdot \FF}$ between Lipschitz functions $\phi$ and divergence-measure fields $\FF$.
Introduced in \cite{Silhavy2009}, this measure is characterised by the product rule
\begin{equation}
  \div(\phi \FF) = \overline{\nabla\phi \cdot \FF}  + \phi \,\div \FF \quad\mbox{in }\mathcal D'(\Omega).
\end{equation} 
We will show this pairing measure admits the following pointwise description and continuity property:

\begin{thm}\label{thm:intro_pairingmeasure}
  Let $\FF \in \mathcal{DM}^{\ext}(\Omega)$ and $\phi \in \WW^{1,\infty}(\Omega)$.
  Then $\overline{\nabla \phi \cdot \FF} \ll \lvert \FF\rvert$ and the associated Radon-Nikod\'ym derivative is given by
  \begin{equation}
    \frac{\d}{\d\lvert \FF\rvert}(\overline{\nabla \phi \cdot \FF})(x) =  \nabla\phi(x) \cdot \frac{\d \FF}{\d \lvert \FF\rvert}(x) \quad\mbox{for $\lvert \FF\rvert$-a.e.\,$x \in \Omega$},
  \end{equation} 
  where $\nabla \phi\cdot v$ denotes the directional derivative of $\phi$ in direction $v \in \mathbb S^{n-1}$, and the relevant derivatives exist $\lvert\FF\rvert$-a.e.
  Furthermore, if $(\phi_k)_k \subset \WW^{1,\infty}(\Omega)$ is a sequence such that $\phi_k \weaksto \phi$ weakly${}^{\ast}$ in $\WW^{1,\infty}(\Omega)$, then the associated pairing measures converge \emph{setwise} in that
  \begin{equation}
    \lim_{k \to \infty} \overline{\nabla \phi_k \cdot \FF}(E) = \overline{\nabla \phi \cdot \FF}(E) \quad\mbox{for all $E \subset \Omega$ Borel.}
  \end{equation} 
\end{thm}
The proof of Theorem~\ref{thm:intro_pairingmeasure} relies on a decomposition of divergence-measure fields into curves, due \textsc{Smirnov} \cite{Smirnov1993}. 
We note that the differentiability statement was already proven by \textsc{Alberti \& Marchese} in \cite{AlbertiMarchese2016}, whose proof also relies on Smirnov's decomposition result.

Theorem~\ref{thm:intro_normaltrace} asserts that the normal trace $\mathrm{N}_E$ maps into a strict subspace of $\Lip_{\rmb}(\partial E)^{\ast}$. We show that this is moreover \emph{optimal}, assuming the following mild regularity condition on the domain.

\begin{defn}\label{defn:intro_lrc}
  We say an open set $U \subset \mathbb R^n$ is \emph{locally uniformly quasiconvex} if there exists $\eps,\delta>0$ such that for any $p,q \in \overline U$ such that $\lvert p - q\rvert < \delta$, there exists a rectifiable curve $\gamma$ connecting $p$ to $q$ through $U$, whose length satisfies $\ell(\gamma)\leq \eps^{-1}\lvert p - q\rvert$.
\end{defn}

\begin{thm}\label{thm:intro_surjective}
  Let $U \subset \mathbb R^n$ be an open set which is locally uniformly quasiconvex. Then
  \begin{equation}
    \mathrm{N}_{U} \colon \mathcal{DM}^{\ext}(U) \twoheadrightarrow\text{\AE}(\partial U) \quad\mbox{is surjective.}
  \end{equation} 
  More precisely, there exists a discrete set $\Lambda \subset U$ such that for each $m \in \text{\AE}(\partial U)$, there exists $\FF \in \mathcal{DM}^{\ext}(U)$ such that $\mathrm{N}_U(\FF) = m$, $\div \FF$ is supported on $\Lambda$, and we have the estimate
  \begin{equation}
    \lVert \FF \rVert_{\mathcal{DM}^{\ext}(U)} \leq C \lVert m \rVert_{\text{\AE}(\partial U)}.
  \end{equation} 
\end{thm}

Thus, for a general class of open sets $U$, we completely characterise the normal trace space on $\partial U$ for divergence-measure fields. As we will discuss in \S\ref{sec:trace_surj} (namely Theorem~\ref{cor:trace_surject_twosided}), we can infer surjectivity results for the trace $N_U \colon \mathcal{DM}^{\ext}(\Omega) \twoheadrightarrow\text{\AE}(\partial U)$ for $U \subset \Omega$ by applying the above on both $U$ and $\overline U^{\rmc}$; here we require not only that $U$ and $\overline{U}^{\rmc}$ satisfies Definition~\ref{defn:intro_lrc}, but also the topological condition $\partial U = \partial \overline U^{\rmc}$.
These conditions are satisfied by all bounded Lipschitz domains, and also by certain fractal domains (see Example \ref{eg:koch_snowflake}).

As a consequence of this characterisation, we obtain extension results for divergence-measure fields; detailed statements can be found in \S\ref{sec:extension}.

\begin{thm}\label{thm:intro_extension}
  Let $U \subset \mathbb R^n$ be a open set such that $\overline{U}^{\rmc}$ is locally uniformly quasiconvex and that $\partial U = \partial \overline U^{\rmc}$. Then there exists a (not necessarily linear) extension operator
  \begin{equation}
    \mathcal{E}_U \colon \mathcal{DM}^{\ext}(U) \to \mathcal{DM}^{\ext}(\mathbb R^n)
  \end{equation} 
  such that $\mathcal{E}_U(\FF) \mres U = \FF$ for all $\FF \in \mathcal{DM}^{\ext}(\Omega)$, and $\mathcal{E}_U$ is bounded in that
  \begin{equation}
    \lVert \mathcal{E}_U(\FF) \rVert_{\mathcal{DM}^{\ext}(\mathbb R^n)} \leq C \lVert \FF \rVert_{\mathcal{DM}^{\ext}(\Omega)} \quad\mbox{for all $\FF \in \mathcal{DM}^{\ext}(\Omega)$}.
  \end{equation} 
\end{thm}

As Example~\ref{eg:extension_counterexample} will illustrate, the condition $\partial U = \partial \overline U^{\rmc}$ is in general necessary.
Similarly as in Theorem~\ref{thm:intro_surjective}, the extension $\widetilde\FF$ we construct satisfies $\div \widetilde\FF = 0$ away from a discrete set $\Lambda$ on the complement.
This allows us to establish extension results for divergence-free fields; for this we let
\begin{equation}\label{eq:divfree_meas}
  \mathcal M_{\div}(U;\mathbb R^n) = \{ \FF \in \mathcal{M}(\Omega;\mathbb R^n) : \div \FF = 0 \} \subset \mathcal{DM}^{\ext}(U).
\end{equation} 

\begin{thm}\label{thm:intro_divfree}
  Let $U \subset \mathbb R^n$ be a open set such that $\overline U^{\rmc}$ is locally uniformly quasiconvex and that $\partial U = \partial \overline U^{\rmc}$.
  Then there exists an open set $\widetilde{U} \subset \mathbb R^n$ such that $\overline U \subset \widetilde U$ and $\dist(U,\partial \widetilde U) > 0$, and an extension operator
  \begin{equation}
    \mathcal E_{U,\widetilde U} \colon \mathcal{M}_{\div}(U;\mathbb R^n) \to \mathcal{M}_{\div}(\widetilde U;\mathbb R^n)
  \end{equation} 
  such that $\mathcal E_{U,\widetilde U}(\FF) \mres U = \FF$ for all $\FF \in \mathcal{M}_{\div}(U;\mathbb R^n)$.
  Furthermore, if $\partial U$ is bounded and $\overline{U}^{\rmc}$ is connected in addition, then we can choose $\widetilde{U} = \mathbb R^n$.
\end{thm}

Finally, we show that the restriction of the trace operator to $\mathcal{DM}^1$ remains surjective; despite being a seemingly stronger statement, this will follow as a consequence of Theorem~\ref{thm:intro_extension}.
This also implies an extension theorem analogous to Theorem~\ref{thm:intro_extension} for fields in $\mathcal{DM}^1(U)$, which we will detail in \S\ref{sec:dm1}.

\begin{thm}
  Let $U \subset \mathbb R^n$ be an open set satisfying Definition~\ref{defn:intro_lrc}.
  Then restriction of the trace operator
  \begin{equation}
    \mathrm{N}_U \colon \mathcal{DM}^1(U) \twoheadrightarrow\text{\AE}(\partial U) \quad\mbox{is surjective.}
  \end{equation} 
  More precisely, for each $m \in \text{\AE}(\partial U)$ there exists $\GG \in \mathcal{DM}^1(U)$ such that $\mathrm{N}_U(\GG)= m$ and the estimate
  \begin{equation}
    \lVert \GG \rVert_{\mathcal{DM}^1(U)} \leq C \lVert m \rVert_{\text{\AE}(\partial U)}
  \end{equation}
  holds. We moreover have that $\div \GG \in \LL^1(U)$.
\end{thm}

\textbf{Organisation}: We recall some necessary results in \S\ref{sec:prelims} regarding divergence-measure fields, the space of Lipschitz functions and its predual, along with Smirnov's decomposition theorem.
A proof of Smirnov's theorem, in the form we use, is also included in Appendix~\ref{sec:smirnov_proof}.
Equipped with these results, \S\ref{sec:anzelotti_rep} is dedicated to the proof of Theorem~\ref{thm:intro_pairingmeasure}, and consequences of said theorem is explored in \S\ref{sec:pairing_applications}, where Theorem~\ref{thm:intro_normaltrace} is proved.
The surjectivity of this refined trace operator is proven in \S\ref{sec:trace_surj},
from which we can infer the extension results, namely Theorems~\ref{thm:intro_extension} and \ref{thm:intro_divfree}, in \S\ref{sec:extension}.
Finally, in \S\ref{sec:dm1} we establish similar surjectivity and extension results in $\mathcal{DM}^1$.

\section{Preliminaries}\label{sec:prelims}

We begin by setting our conventions, and recording some results that will be used in the sequel. We will start with:

\textbf{Notation}: Throughout the paper, we will consider an open set $\Omega \subset \mathbb R^n$ with $n \geq 2$.
Given any set $A \subset \mathbb R^n$ we denote its complement by $A^{\rmc}$ and the associated indicator function by $\mathbbm{1}_A$.
In $\Omega$ we denote the space of finite signed Radon measures by $\mathcal M(\Omega)$, the space of bounded Lipschitz functions by $\Lip_{\rmb}(\Omega)$, and the space of bounded Borel functions by $\mathcal B_{\rmb}(\Omega)$.
Also we denote by $\CC^k_{\rmb}(\Omega)$ as the space of $k$-times continuously differentiable functions with bounded derivatives in $\Omega$, and put $\CC_{\rmb}(\Omega) = \CC_{\rmb}^0(\Omega)$ for the space of bounded continuous functions.

If $\mu$ is a Borel measure on $\Omega$ and $A \subset \Omega$ is a Borel subset, then we will denote by $\mu \mres A$ the Borel measure on $\mathbb R^n$ satisfying $(\mu \mres A)(B) = \mu(A \cap B)$ for any Borel set $B \subset \mathbb R^n$.
We also denote the $n$-dimensional Lebesgue measure and $k$-dimensional Hausdorff measure on $\mathbb R^n$ by $\mathcal L^n$ and $\mathcal H^k$ respectively.

In general for a function space $\mathrm X$, we will write $\mathrm X(\Omega;\mathbb R^n)$ for the space of functions valued in $\mathbb R^n$, and write $\mathrm X_{\rmc}$, $\mathrm X_{\loc}$ for compactly supported functions in $\mathrm X(\Omega)$ and functions locally in $\mathrm X$ respectively.
Also given two Banach spaces $\mathcal X$ and $\mathcal Y$, will write $\mathcal X \simeq \mathcal Y$ if they are isomorphic, and $\mathcal X \cong \mathcal Y$ they are moreover isometrically isomorphic. The dual space of $\mathcal X$ will be denoted $\mathcal X^{\ast}$.

\subsection{Divergence-measure fields}

We recall the central notions of interest in this paper.

\begin{defn}\label{defn:divmeas}
  Let $\Omega \subset \mathbb R^n$ be an open set.
  A measure-valued field $\FF = (F_1,\ldots,F_n) \in \mathcal{M}(\Omega;\mathbb R^n)$ is called a \emph{divergence-measure field} if $\div \FF$ is represented by a finite measure.
  That is, there exists a signed measure $\div \FF \in \mathcal M(\Omega)$ which satisfies
  \begin{equation}
    \int_{\Omega} \nabla \phi \cdot \d\FF = - \int_{\Omega} \phi \,\d(\div \FF) \quad\mbox{for all $\phi \in \CC_{\rmc}^{1}(\Omega)$}.
  \end{equation} 
  The space of divergence-measure fields will be denoted by $\mathcal{DM}^{\ext}(\Omega)$, which we equip with the norm
  \begin{equation}
    \lVert \FF\rVert_{\mathcal{DM}^{\ext}(\Omega)} = \lvert \FF\rvert(\Omega) + \lvert \div \FF\rvert(\Omega).
  \end{equation} 
  If the field $\FF$ is represented by a function in $\LL^p(\Omega)$ with $1 \leq p \leq \infty$, we write $\FF \in \mathcal{DM}^p(\Omega)$.
\end{defn}

In \S \ref{sec:extension} we will also consider divergence-free fields on $\Omega$; the space of such fields will be denoted by $\mathcal{M}_{\div}(\Omega;\mathbb R^n) \subset \mathcal M(\Omega;\mathbb R^n)$.

\begin{remark}
  A \emph{normal $1$-current} in $\Omega$ is a $1$-current $T \in \mathcal{D}_1(\Omega)$ such that both $T$ and $\partial T$ are represented by finite Radon measures in $\Omega$, and the space of such currents will be denoted by $\mathcal{N}_1(\Omega)$.
  Note that in component form we can write
  \begin{equation}
    T = \sum_{i=1}^n T_i \,\d x_i, \quad \partial T = -\sum_{i=1}^n \D_iT_i,
  \end{equation} 
  from which we see that $T \in \mathcal{N}_1(\Omega)$ if and only if the measure-valued field $\bm{T} := (T_1,\ldots,T_n)$ is a divergence-measure field.
  This gives a one-to-one correspondence
  \begin{equation}
    \mathcal{N}_1(\Omega) \cong \mathcal{DM}^{\ext}(\Omega),
  \end{equation} 
  thereby providing a geometric viewpoint which allows us to naturally identify curves as divergence-measure fields (see \S\ref{sec:smirnov}).
\end{remark}

\begin{defn}
  Let $\FF \in \mathcal{DM}^{\ext}(\Omega)$ and $E \subset \Omega$ a Borel subset.
  The \emph{normal trace} of $\FF$ on $\partial E$ is defined as the linear functional
  \begin{equation}
    \langle \FF \cdot \nu, \phi \rangle_{\partial E} = - \int_E \nabla\phi\cdot\d\FF - \int_E \phi\,\d(\div \FF) \quad\mbox{for all $\phi \in \CC_{\rmb}^1(\Omega)$}.
  \end{equation} 
  If the normal trace can be represented as a measure on $\partial E$, we will denote it by $(\FF \cdot \nu)_{\partial E}$.
\end{defn}
We will adopt the convention of taking the inner unit normal as in \cite{ChenEtAl2024}.
By considering the restriction of $\phi \in \CC^1_{\rmc}(\mathbb R^n)$ to $\Omega$, as distributions in $\mathbb R^n$ we have
\begin{equation}\label{eq:normaltrace_distributional}
  \langle \FF \cdot \nu, \,\cdot\,\rangle_{\partial E} = \div(\mathbbm{1}_E\FF) -\mathbbm{1}_E\div \FF,
\end{equation} 
so we see that the normal trace is represented by a measure if and only if $\mathbbm{1}_U \FF \in \mathcal{DM}^{\ext}(\mathbb R^n)$.

We will briefly recall some useful properties of the normal trace.
Given an open set $\Omega \subset \mathbb R^n$, for each $\delta>0$ we set 
\begin{equation}\label{eq:domain_approx}
  \widetilde{\Omega}^{\delta} := \Big\{ x \in \Omega : \lvert x\rvert < \frac1{\delta}, \ \ \dist(x,\partial\Omega)>\delta\Big\}.
\end{equation} 
Observe that $\widetilde\Omega^{\delta_2} \Subset \widetilde\Omega^{\delta_1} \Subset \Omega$ for all $\delta_2 > \delta_1>0$ and $\bigcup_{\delta>0} \widetilde\Omega^{\delta} = \Omega$.

\begin{lem}\label{lem:trace_measure_approx}
  Let $\Omega \subset \mathbb R^n$ be an open set and $\FF \in \mathcal{DM}^{\ext}(\Omega)$. Then for $\mathcal L^1$-a.e.\,$\delta>0$, the normal trace $(\FF \cdot \nu)_{\partial\widetilde\Omega^{\delta}}$ is represented by a measure on $\partial\widetilde\Omega^{\delta}$.
  Moreover, for any such $\delta>0$, we have $\mathbbm{1}_{\widetilde{\Omega}^{\delta}}\FF \in \mathcal{DM}^{\ext}(\Omega)$ with
  \begin{equation}\label{eq:approx_divF}
    \div(\mathbbm{1}_{\widetilde{\Omega}^{\delta}}\FF) = \mathbbm{1}_{\widetilde{\Omega}^{\delta}} \div \FF + (\FF \cdot \nu)_{\partial\widetilde{\Omega}^{\delta}}.
  \end{equation} 
\end{lem}

\begin{proof}
  By \cite[Lem.\,7.3, 7.4]{ChenEtAl2024}, for $\mathcal L^1$-a.e.\,$\delta>0$ we have
  \begin{equation}
    \limsup_{\eps \to 0} \frac1{\eps} \lvert \FF \rvert(\{ x \in \widetilde\Omega^{\delta} : \dist(x,\partial\widetilde\Omega^{\delta}) < \eps)\}) < \infty,
  \end{equation} 
  so by \cite[Thm.\,2.4(ii)]{Silhavy2009} the normal trace of $\FF$ on $\partial\widetilde\Omega^{\delta}$ is a measure, and \eqref{eq:approx_divF} follows from \eqref{eq:normaltrace_distributional}.
\end{proof}

\begin{definition}\label{defn:pairing_measure}
  Let $\FF \in \mathcal{DM}^{\ext}(\Omega)$ and $\phi \in \WW^{1,\infty}(\Omega)$.
  We define the \emph{pairing measure} between $\nabla \phi$ and $\FF$ via
  \begin{equation}\label{eq:pairing_def}
    \overline{\nabla \phi \cdot \FF} := \div(\phi \FF) - \phi \div \FF \quad\text{ in } \mathcal D'(\Omega).
  \end{equation} 
  Moreover, if $\phi \in \CC^1_{\rmb}(\Omega)$, we can understand $\overline{\nabla \phi \cdot \FF} = \nabla \phi \cdot \FF$ in the classical pointwise sense.
\end{definition}

Although $\overline{\nabla \phi \cdot \FF}$ is defined as a distribution in $\Omega$, as the terminology suggests, it is in fact represented by a measure, which is the content of the following lemma.
We refer to \cite[Prop.\,5.2]{Silhavy2008} for a proof (see also \cite[Thm.\,2.7]{ChenEtAl2024}).

\begin{lem}[Product rule, {\cite{Silhavy2009}}]\label{lem:product_rule}
  Let $\FF \in \mathcal{DM}^{\ext}(\Omega)$ and $\phi \in \WW^{1,\infty}(\Omega)$.
  Then $\phi \FF \in \mathcal{DM}^{\ext}(\Omega)$, so $\overline{\nabla\cdot\FF}$ is also a finite measure on $\Omega$ via \eqref{eq:pairing_def}.
  Moreover for any $A \subset \Omega$ open, if $(\phi_k)_k \in \CC_{\rmb}^1(A)$ such that $\phi_k \to \phi\rvert_A$ weakly${}^{\ast}$ in $\WW^{1,\infty}(A)$, then
  \begin{equation}\label{eq:productrule_limit}
    \wslim_{k \to \infty} \nabla \phi_k \cdot \FF \mres A = \overline{\nabla \phi \cdot \FF} \mres A
  \end{equation} 
  as measures.
  We also have
  \begin{equation}\label{eq:pairing_abscont}
    \lvert \overline{\nabla \phi \cdot \FF}\rvert \leq \lVert \nabla\phi\rVert_{\LL^{\infty}(\Omega)}\lvert \FF\rvert
  \end{equation} 
  as measures in $\Omega$.
\end{lem}

Using this pairing measure, we can extend the normal trace to be defined on $\WW^{1,\infty}(\Omega)$ as
\begin{equation}
  \langle \FF \cdot \nu, \phi \rangle_{\partial E} = - \int_{E} \,\d(\overline{\nabla \phi \cdot \FF}) - \int_E \phi \,\d(\div \FF) \quad\mbox{for all $\phi \in \WW^{1,\infty}(\Omega)$.}
\end{equation} 
With this we can formulate the following result from \cite{Silhavy2009} (see also \cite[Lem.\,10.2]{ChenEtAl2024}) we mentioned in the introduction.

\begin{lem}\label{lem:normal_lipdual}
  Let $\FF \in \mathcal{DM}^{\ext}(\Omega)$ and $U \subset \Omega$ be an open set. Then there exists a bounded linear operator
  \begin{equation}
    \mathrm{N}_U(\FF) \colon \mathcal{DM}^{\ext}(\Omega) \to \Lip_{\mathrm{b}}(\partial U)^{\ast}
  \end{equation} 
  which satisfies
  \begin{equation}
    \langle \mathrm{N}_U(\FF), \phi\rvert_{\partial U} \rangle = \langle \FF \cdot \nu, \phi \rangle_{\partial U} \quad\mbox{for all $\phi \in \Lip_{\mathrm{b}}(\Omega)$.}
  \end{equation} 
\end{lem}

\subsection{The predual of the space of Lipschitz functions}
Given any subset $X \subset \mathbb R^n$, we will denote by $\Lip_{\rmb}(X)$ the space of Lipschitz continuous functions $\phi$ on $X$ which are bounded in that
\begin{equation}
\lVert \phi \rVert_{\Lip_{\rmb}(X)} := \max\Big\{\lVert \phi\rVert_{\LL^{\infty}(X)},\,\Lip(\phi,X)\Big\} < \infty.
\end{equation} 
This norm is equivalent to the more standard choice $\lVert \phi\rVert_{\LL^{\infty}(X)}+\Lip(\phi,X)$, however the above form will be more convenient in describing the predual.
We observe that if $\phi \in \Lip_{\rmb}(X)$, where $X \subset \mathbb R^n$ is any set, we have $\phi$ is uniformly continuous and hence admits a unique continuous extension to $\overline X$, which is remains Lipschitz continuous with the same norm. This gives a natural identification $\Lip_{\rmb}(X) \cong \Lip_{\rmb}(\overline X)$.

Note that for $U \subset \mathbb R^n$ open we have $\Lip_{\rmb}(U) \subset \WW^{1,\infty}(U)$; while equality holds for domains satisfying Definition~\ref{defn:intro_lrc} (see \cite[Thm.\,4.1]{Heinonen2005}), the inclusion is in general strict.
We will show that both spaces can be identified as a dual space, and thus can be equipped with the topology of weak${}^{\ast}$-convergence.

\begin{lem}\label{lem:w1infty_predual}
  Let $U \subset \mathbb R^n$ be an open set. Then there is an isometric isomorphism
  $ \WW^{1,\infty}(U) \cong \WW^{-1,1}(U)^{\ast} $
  induced by the pairing
  \begin{equation}\label{eq:w1infty_pairing}
    \langle g,\phi\rangle = \int_{U} f_0 \phi \,\d x  - \sum_{i=1}^n \int_{U} f_i \partial_{x_i}\phi \,\d x, \quad \mbox{for all }\phi \in \WW^{1,\infty}(U),\ g \in \WW^{-1,1}(U),
  \end{equation}
  where $g = f_0 + \sum_{i=1}^n \partial_{x_i}f_i$ with $f_0,f_1,\cdots,f_n \in \LL^1(U)$.
\end{lem}

Here $\WW^{-1,1}(U)$ is equipped with the norm $\lVert g \rVert_{\WW^{-1,1}(U)} = \inf\{\sum_{i=0}^n \lVert f_i \rVert_{\LL^1(U)}\}$, where the infimum is taking over all such representations of $g$.
We refer to \cite[Ex.\,2.3,Rmk.\,3.12]{AmbrosioEtAl2000} for a proof.

From this we see that $\WW^{1,\infty}(U)$ is the dual of a separable space, and we will also use this precise description in \S\ref{sec:dm1}.
For a general set $X \subset\mathbb R^n$, $\Lip_{\rmb}(X)$ also has the structure of a dual space, which was described by \textsc{Arens \& Eells} in \cite{ArensEells1956}; for this we introduce the following definition.

\begin{defn}\label{defn:arenseells}
  Let $X \subset \mathbb R^n$ be a closed set. We define the \emph{Arens-Eells space} $\text{\AE}(X)$ to be the completion of the space
  \begin{equation}
    \text{\AE}_0(X) := \mathrm{span}\{ \delta_x : x \in X \} \subset \Lip_{\rmb}(X)^{\ast}
  \end{equation} 
  with respect to the dual norm on $\Lip(X)^{\ast}$, which we denote by $\lVert \cdot \rVert_{\text{\AE}(X)}$.
  We also denote the induced pairing by $\langle \cdot, \cdot \rangle_{\text{\AE}}$, which satisfies
  \begin{equation}\label{eq:ae_pairing}
    \langle m, \phi \rangle_{\text{\AE}} = \sum_{i=1}^k a_i \phi(p_i), \quad\mbox{where } \ m = \sum_{i=1}^k a_i \delta_{p_i} \in \text{\AE}_0(X),\ \phi \in \Lip_{\rmb}(X).
  \end{equation} 
\end{defn}

We will need a more concrete description in \S \ref{sec:trace_surj};
for this, following \cite[\S 3]{Weaver2018} we also introduce the following related space.
Note that our conventions differ from the aforementioned text, since we do not consider pointed spaces.

\begin{defn}
  Let $X \subset \mathbb R^n$ be closed, and fix $e \notin X$.
  Then writing $\widetilde X = X \cup \{e\}$, define $A(X,e)$ as the completion of the space
  \begin{equation}
    A_0(X,e) = \mathrm{span}\{ \delta_{q} - \delta_{p} : p,q \in \widetilde X \}
  \end{equation} 
  with respect to the norm
  \begin{equation}\label{eq:AE_equivalent_norm}
    \lVert m \rVert_A = \inf\left\{ \sum_{i=1}^k \lvert a_i\rvert \rho(p_i,q_i) : m = \sum_{i=1}^k a_i(\delta_{q_i} - \delta_{p_i}),
    \begin{matrix} a_i \in \mathbb R, \, p_i,q_i \in \widetilde X \\  k \in \mathbb N, \, 1 \leq i \leq k\end{matrix} \right\},
  \end{equation} 
  where $\rho$ is a modified metric on $\widetilde X = X \cup \{e\}$ given by
  \begin{alignat}{3}\label{eq:metric_rho}
    \rho(p,q) &= \min\{\lvert p -q\rvert, 2 \} \quad&&\mbox{for all $p,q \in X$}, \\
    \rho(p,e) &= 1 \quad&&\mbox{for all $p \in X$}.\label{eq:metric_rho2}
  \end{alignat}
  We also denote the induced pairing by $\langle \cdot, \cdot\rangle_A$, which satisfies
  \begin{equation}\label{eq:dense_pairing}
    \langle m, \phi \rangle_A = \sum_{i=1}^k a_i (\phi(q_i)-\phi(p_i)), \quad m = \sum_{i=1}^k a_i (\delta_{q_i}-\delta_{p_i}) \in A_0(X,e), \ \phi \in \Lip_{\rmb}(X),
  \end{equation} 
  understanding that $\phi(e) = 0$.
\end{defn}

\begin{prop}\label{prop:ae_isometric}
  Let $X \subset \mathbb R^n$ be a closed set and $e \notin X$. Then,
  \begin{enumerate}[label=(\alph*)]
    \item\label{item:ae_description} $A(X,e) \cong\text{\AE}(X)$ via $m \mapsto m \mres X$ for $m \in A_0(X,e)$, extending by density,

    \item\label{item:ae_dual} $\Lip_{\rmb}(X) \simeq \text{\AE}(X)^{\ast}$ via $\phi \mapsto (m \mapsto \langle m,\phi\rangle_{\text{\AE}})$.
  \end{enumerate}
\end{prop}

\begin{proof}
  By \cite[Thm.\,3.3, Cor.\,3.4]{Weaver2018}, the pairing \eqref{eq:dense_pairing} extends to an isometric isomorphism $A(X,e)^{\ast} \cong \Lip_{\rmb}(X)$.
  In particular this implies that for $m \in \text{\AE}_0(X,e)$,
  \begin{equation}
    \lVert m \mres X \rVert_{\Lip_{\rmb}(X)^{\ast}} = \sup\{\lvert\langle m,\phi\rangle_{A}\rvert  : \lVert\phi\rVert_{\Lip_{\rmb}(X)} \leq 1 \}= \lVert m \rVert_{A}.
  \end{equation} 
  Hence it follows that $A(X,e) \cong \text{\AE}(X)$ by sending each $m \in A_0(X,e)$ to $m \mres X \in \text{\AE}_0(X)$ and extending by density, thereby proving \ref{item:ae_description}.
  Since we also have $\langle m,\phi\rangle_A = \langle m \mres X,\phi \rangle_\text{\AE}$ for all $m \in A_0(X,e)$ and $\phi \in \Lip_{\rmb}(X)$, it follows
  that $\text{\AE}(X)^{\ast} \cong \Lip_{\rmb}(X)$ via the pairing $\langle \cdot,\cdot \rangle_{\text{\AE}}$, establishing \ref{item:ae_dual}.
\end{proof}

\begin{eg}
  If $X \subset \mathbb R^n$ is closed, we have $\mathcal M(X) \subset \text{\AE}(X)$, by noting for that each $\mu \in \mathcal M(X)$, the mapping $\phi \mapsto \int_X \phi \,\d \mu$ is well-defined and weakly${}^{\ast}$-continuous on $\Lip_{\rmb}(X)$.
  However this space is strictly larger in general; if $a \in X$ is an accumulation point of $X$, then we can find a sequence $(a_k)_k \subset X$ converging to $a$ such that $a_k \neq a$ for all $k$. 
  By passing to a subsequence if necessary, assume that $\sum_k \lvert a_k - a\rvert < \infty$.
  Then $m = \sum_{k=1}^{\infty} (\delta_{a_k} - \delta_{a}) \in \text{\AE}(X)$ by noting the series converges absolutely in $\Lip_{\rmb}(X)^{\ast}$.
\end{eg}

\begin{lem}\label{lem:lipb_conv}
  Let $X \subset \mathbb R^n$ be any set, and $\phi_k, \phi \in \Lip_{\rmb}(X)$.
  Then, as $k \to \infty$,
  \begin{equation}
    \phi_k \weaksto \phi\ \mbox{ in } \Lip_{\rmb}(X)  \iff \begin{cases} \phi_k \to \phi \text{ uniformly on bounded subsets of $X$,} \\ \sup_k \lVert \phi_k\rVert_{\Lip_{\rmb}(X)} < \infty.  \end{cases}
  \end{equation} 
\end{lem}

\begin{proof}
  Using the identification $\Lip_{\rmb}(X) \cong \Lip_{\rmb}(\overline X)$, we can assume without loss of generality that $X$ is closed.
  If $\phi_k \weaksto \phi$ weakly${}^{\ast}$ in $\Lip_{\rmb}(X)$, by the Banach-Steinhaus theorem, we have $\lVert \phi_k \rVert_{\Lip_{\rmb}(X)}$ is uniformly bounded in $k$.
  Then by applying the Arzel\`a-Ascoli theorem, there is a subsequence $\phi_{k_j}$ which converges uniformly to $\phi$ on $X \cap B_M(0)$ for each $M \in \mathbb N$, and hence $\phi_{k_j} \to \phi$ uniformly on bounded subsets of $X$.
  Since the limit is unique, this convergence also holds for the entire sequence $\phi_k$.
  
  Conversely since $\Lip_{\rmb}(X)$ is the dual of a separable space, the weak${}^{\ast}$-topology is compact and metrisable on norm-bounded subsets (see \emph{e.g.}\,\cite[Thm.\,3.16, 3.28]{Brezis2011}).
  Therefore $\phi_k$ admits a weakly${}^{\ast}$-convergent subsequence, but since this limit is uniquely determined as $\phi$ using the uniform convergence, the entire sequence $\phi_k$ converges weakly${}^\ast$ to $\phi$.
\end{proof}

We will often use Lemma~\ref{lem:lipb_conv} with $X = [0,1]$, noting that $\Lip_{\rmb}([0,1]) = \WW^{1,\infty}((0,1))$. For general open sets $U$ however, we have a slightly different characterisation for weak${}^{\ast}$ convergence in $\WW^{1,\infty}(U)$.

\begin{lem}\label{lem:w1infty_conv}
  Let $U \subset \mathbb R^n$ be an open set. Then if $\phi_k, \phi \in \WW^{1,\infty}(U)$, as $k \to \infty$,
  \begin{equation}\label{eq:w1infty_conv}
  \phi_k \weaksto \phi\ \mbox{ in } \WW^{1,\infty}(U)  \iff \begin{cases} \phi_k \to \phi \text{ pointwise,} \\ \sup_k \lVert \phi_k\rVert_{\WW^{1,\infty}(U)} < \infty.  \end{cases}
\end{equation} 
In addition, the space $\CC^1_{\rmb}(U)$ is sequentially weakly${}^{\ast}$ dense in $\WW^{1,\infty}(U)$.
\end{lem}

We note that $\CC^1_{\rmb}(U) \not\subset \Lip_{\rmb}(U)$ in general, so we do not get an analogous density statement there. Also the below proof shows in fact that $\phi_k \to \phi$ locally uniformly in $U$ in \eqref{eq:w1infty_conv}, however pointwise convergence will suffice for our purposes.

\begin{proof}
  The equivalence \eqref{eq:w1infty_conv} can be proven analogously as in Lemma~\ref{lem:lipb_conv}, noting that $\WW^{-1,1}(U)$ is also separable. One difference lies in showing the pointwise convergence; for this assume that $\phi_k \weaksto \phi$ and fix $x \in U$.
  Then there is a neighbourhood $B_r(x) \subset U$, and since $\phi_k \rvert_{B_r(x)}$ is bounded in $\WW^{1,\infty}(B_r(x)) = \Lip_{\rmb}(B_r(x))$, we infer that $\phi_k$ converges uniformly to $\phi$ on this ball $B_r(x)$.

  To show $\CC^1_{\rmb}(U)$ is sequentially weakly${}^{\ast}$ dense, given $\phi \in \WW^{1,\infty}(U)$ we will take $\phi_k$ to be as in the construction from \cite[Thm.\,4.2]{EvansGariepy2015}.
  More precisely, given a covering $\{V_j\}_{j=1}^{\infty}$ of $U$ such that $V_j \Subset U$ for each $j$, let $\{\zeta_j\}$ be a partition of unity subordinate to $\{V_j\}$.
  We then let $\phi_k = \sum_j \eta_{\eps_{j,k}} \ast (\zeta_j\phi )$, where $\eta_{\eps}$ is a standard mollifier and $\eps_{j,k}$ is chosen to satisfy $\eps_{j,k} \leq \dist(\spt(\zeta_j),\partial V_j)$ and that $\eps_{j,k} \searrow 0$ as $k \to \infty$ for each $j$.
  Since each $\phi \nabla\zeta_j$ is continuous, by shrinking $\eps_{j,k}$ if necessary we can also assume that
  \begin{equation}
    \lVert \eta_{\eps_{j,k}} \ast(\phi\nabla\zeta_j) - \phi\nabla\zeta_j \rVert_{\LL^{\infty}(U)} \leq 2^{-j}
  \end{equation} 
  for all $j,k \in \mathbb N$.
  Then noting that $\sum_{j=1}^{\infty} \nabla\zeta_j=\nabla \mathbbm{1}_U =0$ in $U$, we can estimate
  \begin{equation}
    \Big\lVert \sum_{j=1}^{\infty} \eta_{\eps_{j,k}} \ast (\phi \nabla\zeta_j) \Big\rVert_{\LL^{\infty}(U)} \leq \sum_{j=1}^{\infty} \lVert \eta_{\eps_{j,k}} \ast(\phi\nabla\zeta_j) - \phi\nabla\zeta_j \rVert_{\LL^{\infty}(U)} \leq 1
  \end{equation} 
  for all $k$.
  We can then verify that $(\phi_k)_k$ is uniformly bounded in $\WW^{1,\infty}(U)$ and converges pointwise to $\phi$ as $k \to \infty$, so by \eqref{eq:w1infty_conv} we infer that $\phi_k \weaksto \phi$ in $\WW^{1,\infty}(U)$.
\end{proof}

\subsection{The Smirnov decomposition}\label{sec:smirnov}

We will state a version of Smirnov's decomposition theorem, valid for fields in $\mathcal{DM}^{\ext}(\Omega)$.
For this, we first define the space of curves we will work with, namely
\begin{equation}
  \mathscr C_1 = \mathscr C_1^n := \{ \gamma \in \Lip_{\rmb}([0,1];\mathbb R^n) : \Lip(\gamma) \leq 1 \},
\end{equation} 
equipped with the topology of uniform convergence.
We will refer to elements $\gamma \in \mathscr{C}_1$ as \emph{curves}, and we say $\gamma$ is \emph{closed} if $\gamma(0) = \gamma(1)$.
It is well known  that rectifiable curves admit an arclength reparametrisation, and as the proof of Smirnov's theorem we outline in Appendix~\ref{sec:smirnov_proof} will show, for our purposes it will suffice to consider constant-speed curves in $\mathscr{C}_1$.

Observe that $\mathscr{C}_1$ is locally compact by the Arzel\`a-Ascoli theorem, and is moreover metrisable since the topology is induced by the uniform norm $\lVert \gamma \rVert_{\LL^{\infty}([0,1])}$.
Moreover by Lemma~\ref{lem:lipb_conv}, convergence in $\mathscr{C}_1$ is equivalent to weak${}^{\ast}$ convergence in $\Lip_{\rmb} \simeq \WW^{1,\infty}$.

\begin{definition}\label{defn:curve_current}
  For a curve $\gamma \in \mathscr C_1$, we denote by $\llbracket\gamma\rrbracket \in \mathcal{DM}^{\ext}(\mathbb R^n)$ the field defined to satisfy
  \begin{equation}\label{eq:curve_current}
    \langle \llbracket\gamma\rrbracket, \Phi \rangle  = \int_0^1 \Phi(\gamma(t)) \cdot \gamma^\prime(t) \,\d t= \int_{\mathbb R^n} \bigg( \sum_{t \in \gamma^{-1}(p)} \Phi(p) \cdot \frac{\gamma^\prime(t)}{\lvert\gamma^\prime(t)\rvert} \bigg) \,\d\mathcal{H}^1(p) 
  \end{equation} 
  for all $\Phi \in \mathcal{B}_{\rmb}(\mathbb R^n;\mathbb R^n)$, where the latter equality follows by the area formula (see \emph{e.g.}\,\cite[Thm.\,2.71, (2.47)]{AmbrosioEtAl2000}).
  Observe this satisfies
  \begin{equation}
    \div \llbracket\gamma\rrbracket = \delta_{\gamma(1)} - \delta_{\gamma(0)},
  \end{equation} 
  which is zero if and only if $\gamma$ is closed.
  We will also denote the total variation measure of $\llbracket \gamma \rrbracket$ by $\mu_{\llbracket\gamma\rrbracket}$.
\end{definition}

Note that, if for some $\Omega \subset \mathbb R^n$ open we have $\gamma(t) \in \Omega$ for all $t \in (0,1)$, then the associated field lies in $\mathcal{DM}^{\ext}(\Omega)$.
Also if we set $\Gamma_{\gamma} = \gamma([0,1])$ and let
\begin{equation}
  \xi_{\gamma}(p) := \begin{cases}
   \sum_{t \in \gamma^{-1}(p)} \frac{\gamma^\prime(t)}{\lvert \gamma^\prime(t)\rvert} &\text{if } p \in \Gamma_{\gamma}, \\
    0 &\text{otherwise},
  \end{cases}
\end{equation} 
which is defined $\mathcal{H}^1$-a.e.\,on $\Gamma_{\gamma}$, then by \eqref{eq:curve_current} we obtain the representation
\begin{align}
  \llbracket\gamma\rrbracket &= \ \xi_{\gamma}(x) \ \mathcal H^1 \mres \Gamma_{\gamma},\\
  \mu_{\llbracket\gamma\rrbracket} &= \lvert\xi_{\gamma}(x)\rvert\,\mathcal H^1 \mres \Gamma_{\gamma}.
\end{align} 
We will denote the \emph{length} of the curve as $\ell(\gamma) := \mu_{\llbracket\gamma\rrbracket}(\Gamma_{\gamma})$; note this corresponds to the mass of $\llbracket\gamma\rrbracket$ as current, and hence accounts for orientation.

Equipped with this terminology, we can state a version of Smirnov's decomposition theorem valid in the full space.

\begin{thm}[Smirnov, {\cite{Smirnov1993}}]\label{thm:decomposition_fullspace}
  Given $\FF \in \mathcal{DM}^{\ext}(\mathbb R^n)$, there exits a non-negative and finite Borel measure $\nu$ on $\mathscr C_{1}$ such that
  \begin{alignat}{3}\label{eq:decomposition_main}
    \int_{\mathbb R^n} \Phi \cdot \d \FF&=\int_{\mathscr C_1} \langle\llbracket \gamma \rrbracket,\Phi \rangle \,\d \nu(\gamma) &&\quad\mbox{for all $\Phi \in \mathcal B_{\rmb}(\mathbb R^n;\mathbb R^n)$},\\
  \label{eq:decomposition_tv}
    \int_{\mathbb R^n} \phi \,\d\lvert \FF\rvert &= \int_{\mathscr C_1} \langle \mu_{\llbracket\gamma\rrbracket},\phi\rangle \,\d \nu(\gamma) &&\quad\mbox{for all $\phi \in \mathcal B_{\rmb}(\mathbb R^n)$}.
  \end{alignat} 
  Here we will have the maps $\gamma \to \langle \llbracket \gamma \rrbracket, \Phi \rangle$ and $\gamma \to \langle \mu_{\llbracket \gamma\rrbracket},\phi \rangle$ are Borel measurable with respect to the topology of uniform convergence, ensuring that these integrals are well-defined.
  Moreover for $\nu$-almost every $\gamma \in \mathscr C_1$, we have $\lvert \gamma^\prime(t)\rvert$ is constant $\mathcal L^1$-a.e.\,on $[0,1]$.
\end{thm}

Since this formulation differs somewhat from what is proven in \cite{Smirnov1993}, a proof is provided in Appendix~\ref{sec:smirnov_proof}.
We point out that the decomposition we obtain is \emph{incomplete} in the sense that
\begin{equation}
  \lvert \div \FF\rvert(\mathbb R^n) \neq \int_{\mathscr C_1} \lvert \div \llbracket \gamma \rrbracket\rvert(\mathbb R^n) \,\d \nu(\gamma)
\end{equation} 
in general, however the formulation we state is technically simpler as we can work with $\mathscr C_1$ as our space of admissible curves.

Since we wish to apply this to fields $\FF \in \mathcal{DM}^{\ext}(\Omega)$, we will need a suitable variant of Theorem~\ref{thm:decomposition_fullspace} valid for domains.
This can be obtained as a consequence of the full-space decomposition as follows.

\begin{thm}\label{thm:decomposition_domain}
  Let $\Omega \subset \mathbb R^n$ be an open set and $\FF \in \mathcal{DM}^{\ext}(\Omega)$.
  Then there exists a non-negative Borel measure $\nu$ on $\mathscr C_1$ such that
  \begin{alignat}{3}\label{eq:decomposition_domain}
    \int_{\Omega} \Phi \cdot \d\FF &= \int_{\mathscr C_1} \langle \llbracket\gamma\rrbracket,\Phi\rangle \,\d\nu(\gamma) &&\quad\mbox{for all } \Phi \in \mathcal B_{\rmb}(\Omega;\mathbb R^n), \\
  \label{eq:decomposition_domain_tv}
    \int_{\Omega} \phi\, \d\lvert\FF\rvert &= \int_{\mathscr C_1} \langle \mu_{\llbracket\gamma\rrbracket},\phi\rangle \,\d\nu(\gamma) &&\quad\mbox{for all } \phi \in \mathcal B_{\rmb}(\Omega).
  \end{alignat} 
  In particular, we have
  \begin{equation}\label{eq:decomposition_domain_tv_2}
    \int_{\mathscr{C}_1} \ell(\gamma)\,\d\nu(\gamma) = \lvert \FF\rvert(\Omega) < \infty.
  \end{equation} 
  Moreover $\nu$-almost every $\gamma \in \mathscr{C}_1$ is supported in $\Omega$ and is such that $\lvert \gamma^\prime(t)\rvert$ is $\mathcal L^1$-a.e.\,constant on $[0,1]$.
\end{thm}

\begin{proof}
  For each $\delta>0$ let $\widetilde\Omega^{\delta}$ be as in \eqref{eq:domain_approx},
  then by Lemma~\ref{lem:trace_measure_approx} we know for $\mathcal L^1$-a.e.\,$\delta>0$ that
  \begin{enumerate}[label=(\roman*)]
    \item\label{eq:approx_prop1} the normal trace $(\FF \cdot \nu)_{\partial \widetilde\Omega^{\delta}}$ is represented by a measure on $\partial\Omega$,
    \item\label{eq:approx_prop2} $\lvert \FF\rvert(\partial\widetilde\Omega^{\delta}) = \lvert \div \FF\rvert(\partial\widetilde\Omega^{\delta}) = 0$,
  \end{enumerate}
  noting that \ref{eq:approx_prop2} holds for all but countably many $\delta$.
  We then let $\delta_k \searrow 0$ such that each $\delta_k$ satisfies the above two properties
  and define
 \begin{equation}
   \FF_k = \mathbbm{1}_{A_k}\FF \quad \text{where } A_k := \widetilde\Omega^{\delta_k} \setminus \widetilde\Omega^{\delta_{k-1}}
  \end{equation} 
  for each $k$, understanding that $\widetilde\Omega^{\delta_{0}} = \varnothing$. Then by properties \ref{eq:approx_prop1}, \ref{eq:approx_prop2} and \cite[Rmk.\,2.12, Thm.\,10.5]{ChenEtAl2024}, we have each $\FF_k \in \mathcal{DM}^{\ext}(\Omega)$ is compactly supported in $\Omega$ and satisfies 
  \begin{equation}
    \div \FF_k = \mathbbm{1}_{A_k}\div \FF + (\FF \cdot \nu)_{\partial\widetilde\Omega^{\delta_k}} - (\FF \cdot \nu)_{\partial\widetilde\Omega^{\delta_{k-1}}}.
  \end{equation} 
  Also since the sets $(A_k)_{k=1}^{\infty}$ have pairwise disjoint support,
  \begin{equation}\label{eq:decomp_tv_preservation}
    \lvert \FF\rvert = \sum_{k=1}^{\infty} \lvert \FF_k\rvert \quad\text{as measures in } \Omega.
  \end{equation} 
  Now by applying Theorem~\ref{thm:decomposition_fullspace} to each $\FF_k$, we obtain Borel measures $\nu_k$ on $\mathscr C_1$ such that \eqref{eq:decomposition_main}, \eqref{eq:decomposition_tv} holds for each $\FF_k$ with $\nu_k$.
  Observe by \eqref{eq:decomposition_tv} that $\nu_k$-a.e.\,curve $\gamma$ is supported on $\overline{A_k}$ for each $k$, so it follows that $\spt(\nu_k) \cap \spt(\nu_{\ell}) \neq \varnothing$ if and only if $\lvert k - \ell\rvert \leq 1$.
  Moreover for each $k$, by \eqref{eq:decomposition_tv} and property \ref{eq:approx_prop2} we can estimate
  \begin{equation}\label{eq:overlap_negligible}
    \int_{\spt(\nu_k) \cap \spt(\nu_{k+1})} \ell(\gamma) \,\d \nu_k(\gamma) \leq \lvert \FF_k\rvert(\overline{A_k} \cap \overline{A}_{k+1}) \leq \lvert \FF\rvert(\partial \widetilde\Omega^{\delta_k}) = 0,
  \end{equation} 
  and similarly for the $\nu_{k+1}$-integral.
  Hence we can define
  \begin{equation}
    \nu = \sum_{k=1}^{\infty} \nu_k \quad\text{on } \mathscr{C}_1,
  \end{equation} 
  which is a well-defined Borel measure satisfying
  \begin{equation}
    \int_{\mathscr{C}_1} \ell(\gamma) \,\d\nu = \sum_{k=1}^{\infty} \int_{\mathscr{C}_1} \ell(\gamma) \,\d\nu_k = \sum_{k=1}^{\infty} \lvert \FF_k\rvert(A_k) = \lvert \FF\rvert(\Omega)  < \infty,
  \end{equation} 
  noting the supports are essentially disjoint by \eqref{eq:overlap_negligible}.
  Using the above with \eqref{eq:decomposition_main} applied to each $\FF_k$ and the dominated convergence theorem, we have for all $\Phi \in \mathcal B_{\rmb}(\Omega;\mathbb R^n)$ that
  \begin{equation}
    \int_{\Omega} \Phi \cdot \d \FF= \sum_{k=1}^{\infty} \int_{\Omega} \Phi \cdot \d \FF_k= \sum_{k=1}^{\infty} \int_{\mathscr C_1} \langle \llbracket \gamma \rrbracket, \Phi \rangle \,\d\nu_k = \int_{\mathscr C_1} \langle \llbracket \gamma \rrbracket, \Phi \rangle \,\d\nu,
  \end{equation} 
  establishing \eqref{eq:decomposition_domain}.
  Similarly for any $\phi \in \mathcal B_{\rmb}(\Omega)$, using \eqref{eq:decomposition_tv} and \eqref{eq:decomp_tv_preservation} we have
  \begin{equation}
    \int_{\Omega} \phi\,\d \lvert\FF\rvert= \sum_{k=1}^{\infty} \int_{\Omega} \phi \, \d\lvert\FF_k\rvert= \sum_{k=1}^{\infty} \int_{\mathscr C_1} \langle \mu_{\llbracket \gamma \rrbracket}, \phi \rangle \,\d\nu_k = \int_{\mathscr C_1} \langle \mu_{\llbracket \gamma \rrbracket}, \phi \rangle \,\d\nu,
  \end{equation} 
  establishing \eqref{eq:decomposition_domain_tv} as required.
\end{proof}
\section{Properties of the Anzellotti pairing}\label{sec:anzelotti_pairing}

\subsection{Representation of the pairing}\label{sec:anzelotti_rep}

In this section we prove Theorem~\ref{thm:intro_pairingmeasure}, along with a representation of the pairing in terms of the Smirnov decomposition.
Recall the pairing measure $\overline{\nabla \phi \cdot \FF}$ was defined in Definition~\ref{defn:pairing_measure} and satisfies $\overline{\nabla \phi \cdot \FF} \ll \lvert \FF\rvert$  by Lemma~\ref{lem:product_rule}.

\begin{thm}\label{thm:pairing_pointwise}
  Given $\FF \in \mathcal{DM}^{\ext}(\Omega)$ and $\phi \in \WW^{1,\infty}(\Omega)$, the directional derivatives
  \begin{equation}
    \nabla \phi (x) \cdot \frac{\d\FF}{\d \lvert \FF\rvert} (x) \quad\text{exists for } \lvert\FF\rvert\text{-a.e.\,} x \in \Omega,
  \end{equation} 
  and the Radon-Nikod\'ym derivative of $\overline{\nabla\phi\cdot\FF}$ with respect to $\lvert \FF\rvert $ is given by
  \begin{equation}\label{eq:precise_product}
    \frac{\d}{\d \lvert \FF\rvert}(\overline{\nabla \phi \cdot \FF})(x) =  \nabla \phi (x) \cdot \frac{\d\FF}{\d \lvert \FF\rvert} (x) \quad\mbox{for } \lvert \FF\rvert\text{-a.e.\,} x \in \Omega
  \end{equation} 
  Moreover decomposing $\FF$ as in Theorem~\ref{thm:decomposition_domain},
   we have
  \begin{equation}\label{eq:pairing_disint}
    \int_{\Omega} \psi\, \d(\overline{\nabla\phi \cdot \FF}) = \int_{\mathscr{C}_1} \int_{\Gamma_{\gamma}} \psi \nabla\phi \cdot \xi_{\gamma} \,\d\mathcal H^1 \,\d \nu(\gamma) \quad\text{for all } \psi \in \mathcal B_{\rmb}(\Omega).
  \end{equation} 
\end{thm}

As a consequence, we infer the following improved continuity property for this pairing.

\begin{thm}\label{thm:setwise_conv}
  Let $\FF \in \mathcal{DM}^{\ext}(\Omega)$ and $(\phi_k)_k \subset \WW^{1,\infty}(\Omega)$ such that $\phi_k \weaksto \phi$ weakly${}^{\ast}$ in $\WW^{1,\infty}(\Omega)$. Then
  \begin{equation}\label{eq:setwise_claim}
    \lim_{k \to \infty} \int_E \d(\overline{\nabla \phi_k \cdot \FF}) = \int_E \d(\overline{\nabla \phi \cdot \FF}) \quad\mbox{for all Borel sets $E \subset \Omega$.}
  \end{equation} 
 That is, the pairing measure converges \emph{setwise} with respect to weak${}^{\ast}$ convergence in $\WW^{1,\infty}(\Omega)$.
\end{thm}

\begin{eg}
We will show, by means of a simple example, that we cannot expect a similar continuity statement in $\FF$; for this consider 
\begin{alignat}{3}
  \FF_k &= \mathrm{e}_1 \mathcal H^1 \mres \Gamma_k, \quad &\Gamma_k  &= \{ x \in \mathbb R^2 : x_2 = 1/k \},  \\
  \FF &= \mathrm{e}_1 \mathcal H^1 \mres \Gamma, \quad & \Gamma  &= \left\{ x \in \mathbb R^2 : x_2 = 0 \right\},
\end{alignat}
which lie $ \mathcal{DM}_{\mathrm{loc}}^{\ext}(\mathbb R^2)$, where $\mathrm{e}_1,\mathrm{e}_2$ are the standard basis vectors of $\mathbb R^2$.
For any $R>0$, observe that $\FF_k, \FF$ are uniformly bounded in $\mathcal{DM}^{\ext}(B_R)$ and that $\FF_k \weaksto \FF$ weakly${}^{\ast}$ in $\mathcal{DM}^{\ext}(B_R)$.
Now taking $E = (0,1)^2$, for any $\phi \in \WW^{1,\infty}(\mathbb R^2)$ we have
\begin{equation}
  \lim_{k \to \infty} \int_E \d(\overline{\nabla\phi \cdot \FF_k}) = \lim_{k \to \infty}\int_0^1 \partial_{x_1}\phi(t,1/k) \,\d t = \phi(1,0) - \phi(0,0),
\end{equation} 
whereas $\lvert \FF \rvert(E) = 0$. 
Therefore taking any $\phi$ such that $\phi(1,0) \neq \phi(0,0)$ we see that
\begin{equation}
  \lim_{k \to \infty} \int_E \d(\overline{\nabla\phi \cdot \FF_k}) = \phi(1,0)-\phi(0,0) \neq 0 = \int_E \d(\overline{\nabla \phi \cdot \FF}),
\end{equation} 
thereby exhibiting failure of continuity with respect to weak${}^{\ast}$ convergence of measures.
\end{eg}

The proofs of Theorems~\ref{thm:pairing_pointwise} and \ref{thm:setwise_conv} result will rely on Smirnov's theorem in the form of Theorem~\ref{thm:decomposition_domain}, and two lemmas.
The first is a differentiability statement from \cite{AlbertiMarchese2016}, which we will use in the following form. 
In what follows, $\mathrm{Gr}(\mathbb R^n) = \bigcup_{k=0}^n \mathrm{Gr}_k(\mathbb R^n)$ denotes the set of all subspaces of $\mathbb R^n$.

\begin{lem}\label{lem:directional_diff}
  Let $\FF \in \mathcal{DM}^{\ext}(\Omega)$ and $\phi \in \WW^{1,\infty}(\Omega)$.
  Then there exists a mapping
  \begin{equation}
    x \mapsto V_{\FF}(x) \colon \Omega \to \mathrm{Gr}(\mathbb R^n)
  \end{equation} 
  such that for $\lvert\FF\rvert$-a.e.\,$x \in\Omega$, $\phi$ is differentiable at $x$ with respect to $V_{\FF}(x)$,
  and if we define the differential
  \begin{equation}
    x \mapsto \D_{\FF}\phi(x) \colon \Omega \to (\mathbb R^n)^{\ast}
  \end{equation} 
  by defining $\D_{\FF}\phi(x)\rvert_{V_\FF}$ to be this derivative and setting $\D_{\FF}\phi(x) \rvert_{V_{\FF}^{\perp}} \equiv 0$, this mapping is Borel measurable.
  Moreover taking the decomposition from Theorem~\ref{thm:decomposition_domain}, for $\nu$-a.e.\,$\gamma\in\mathscr C_1$, where the null set is Borel measurable,
  \begin{equation}\label{eq:curve_tangents_diff}
    \gamma^\prime(t) \in V_{\FF}(\gamma(t))\  \text{ and } \ \D_{\FF}\phi(\gamma(t))(\gamma^\prime(t)) = (\phi \circ \gamma)'(t) \quad\mbox{for $\mathcal L^1$-a.e.\,$t\in[0,1]$}.
  \end{equation} 
\end{lem}

We then denote the associated gradient as $x \mapsto \nabla_{\FF}\phi(x)$, which is defined $\lvert \FF\rvert$-a.e., takes values in $V_{\FF}(x)$, and is Borel measurable as a map $\Omega \to \mathbb R^n$.

\begin{proof}
  We apply Theorem \ref{thm:decomposition_domain} to $\FF$, and extending $\FF$ by zero to $\mathbb R^n$ we see this decomposition remains valid in the full space.
  Working in $\mathbb R^n$, we take $x \mapsto V_{\FF}(x)$ to be the decomposability bundle associated to this decomposition, as defined in \cite[\S 2.6]{AlbertiMarchese2016}, which satisfies $\eqref{eq:curve_tangents_diff}_1$.
  Then $\phi$ is differentiable with respect to $V_{\FF}(x)$ at $\lvert \FF\rvert$-a.e.\,$x$ by \cite[Cor.\,3.9]{AlbertiMarchese2016}, and the measurability of $\D_{\FF}\phi$ follows from \cite[Lem.\,3.6]{AlbertiMarchese2016}.
\end{proof}

The second lemma asserts that Theorems~\ref{thm:pairing_pointwise} and \ref{thm:setwise_conv} hold when $\FF$ is a curve.

\begin{lem}\label{lem:pairing_curve}
  Let $\gamma \in \mathscr{C}_1$ and consider the associated divergence-measure field
  \begin{equation}
    \llbracket\gamma\rrbracket = \xi_{\gamma} \, \mathcal{H}^1 \mres \Gamma_{\gamma}.
  \end{equation} 
  Then for any $\Omega \subset \mathbb R^n$ such that $\Gamma_{\gamma} \subset \Omega$, the following holds:
  \begin{enumerate}[label=(\alph*)]

    \item\label{item:curve_measure_rep} 
      For any $\phi \in \WW^{1,\infty}(\Omega)$,
    \begin{equation*}
      \overline{\nabla \phi \cdot \llbracket\gamma\rrbracket} = (\nabla_{\llbracket\gamma\rrbracket} \phi \cdot \xi_{\gamma}) \,\mathcal H^1 \mres \Gamma_{\gamma} \quad\mbox{as measures}.
    \end{equation*} 

  \item\label{item:curve_measure_conv} 
    If $(\phi_k)_k \subset \WW^{1,\infty}(\Omega)$ such that $\phi_k \weaksto \phi$ weakly${}^{\ast}$ in $\WW^{1,\infty}(\Omega)$, then
  \begin{equation*}
    \overline{\nabla\phi_k \cdot \llbracket\gamma\rrbracket} \rightharpoonup \overline{\nabla\phi \cdot \llbracket\gamma\rrbracket} \quad\text{setwise in } \mathcal M(\Omega).
  \end{equation*} 
  \end{enumerate}
\end{lem}

\begin{proof}
  For \ref{item:curve_measure_rep}, let $\phi \in \WW^{1,\infty}(\Omega)$. 
  Then by Lemma~\ref{lem:w1infty_conv}, there exists a sequence $(\phi_k)_k \subset \CC_{\rmb}^1(\Omega)$ converging weakly${}^{\ast}$ to $\phi$ in $\WW^{1,\infty}(\Omega)$, so in particular $\phi_k \to \phi$ pointwise in $\Omega$ and $M := \sup_k \lVert \phi_k \rVert_{\WW^{1,\infty}(\Omega)} < \infty$.
  Then for each $k$ and $\psi \in \mathcal B_{\rmb}(\Omega)$, noting that $\psi \nabla \phi_k \in \mathcal B_{\rmb}(\Omega;\mathbb R^n)$, we have by Definition~\ref{defn:curve_current} that
  \begin{equation}\label{eq:curve_pairing_1}
    \langle \llbracket\gamma\rrbracket, \psi \nabla \phi_k\rangle  = \int_0^1 \psi \circ \gamma(t) \, \frac{\d}{\d t} \left(\phi_k \circ \gamma\right)(t) \,\d t.
  \end{equation} 
  Since $\phi_k \circ \gamma \to \phi \circ \gamma$ pointwise in $[0,1]$ and 
  \begin{equation}\label{eq:phi_circ_gamma_bounded}
  \lVert (\d/\d t)(\phi_k\circ\gamma)\rVert_{\LL^{\infty}((0,1))} \leq M \lVert\gamma^\prime\rVert_{\LL^{\infty}((0,1))} \quad\text{for all } k,
  \end{equation}
  we infer that $\phi_k \circ \gamma \weaksto \phi \circ \gamma$ weakly${}^{\ast}$ in $\WW^{1,\infty}((0,1))$.
  Hence passing to the limit in \eqref{eq:curve_pairing_1},
  \begin{equation}\label{eq:curve_equality_a}
    \int_{\Omega} \psi \,\d(\overline{\nabla \phi \cdot \llbracket\gamma\rrbracket}) = \lim_{k \to \infty} \langle \llbracket\gamma\rrbracket, \psi \nabla \phi_k \rangle = \int_0^1 \psi \circ \gamma(t)\,\frac{\d}{\d t}( \phi \circ \gamma)(t) \,\d t,
  \end{equation} 
  where we also used \eqref{eq:productrule_limit} for the first equality.
  By Lemma~\ref{lem:directional_diff} we have $(\phi \circ \gamma)'(t) = \nabla_{\llbracket\gamma\rrbracket}\phi(\gamma(t)) \cdot \gamma^\prime(t)$ for $\mathcal L^1$-a.e.\,$t \in (0,1)$, so by the area formula
  \begin{equation}
    \begin{split}
    \int_{\Omega} \psi \,\d(\overline{\nabla\phi \cdot \llbracket\gamma\rrbracket}) 
    &= \int_0^1 \psi \circ \gamma(t) \nabla_{\llbracket\gamma\rrbracket} \phi(\gamma(t))\cdot\gamma^\prime(t) \,\d t
    = \int_{\Gamma_{\gamma}} \psi \nabla_{\llbracket\gamma\rrbracket} \phi \cdot \xi_{\gamma} \,\d\mathcal{H}^1.
    \end{split}
  \end{equation} 
 Since $\psi \in \mathcal B_{\rmb}(\Omega)$ was arbitrary, this establishes \ref{item:curve_measure_rep}.

  For \ref{item:curve_measure_conv}, let $\phi_k \weakstarto \phi$ in $\WW^{1,\infty}(\Omega)$ and fix $\psi \in \mathcal B_{\rmb}(\Omega)$.
  Then since $\phi_k \to \phi$ pointwise in $\Omega$ and $\nabla\phi_k$ is uniformly bounded in $\LL^{\infty}(\Omega)$ by Lemma~\ref{lem:w1infty_conv}, arguing analogously as in \eqref{eq:phi_circ_gamma_bounded} we have $\phi_k \circ \gamma \weaksto \phi \circ \gamma$ weakly${}^{\ast}$ in $\WW^{1,\infty}((0,1))$.
  In particular,
  \begin{equation}\label{eq:1d_weak_conv}
    \frac{\d}{\d t}( \phi_k \circ \gamma) \weaksto \frac{\d}{\d t}( \phi \circ \gamma) \quad\mbox{weakly${}^{\ast}$ in $\LL^{\infty}((0,1))$}.
  \end{equation} 
  Hence using \eqref{eq:curve_equality_a} with $\phi_k$ and sending $k \to \infty$ using \eqref{eq:1d_weak_conv},
  \begin{equation}
    \begin{split}
      \lim_{k \to \infty} \int_{\Omega} \psi \, \d(\overline{\nabla \phi_k \cdot \llbracket\gamma\rrbracket}) 
    &=\lim_{k \to \infty} \int_0^1 \psi \circ \gamma(t)\,\frac{\d}{\d t}( \phi_k \circ \gamma)(t) \,\d t \\
    &=\int_0^1 \psi \circ \gamma(t)\,\frac{\d}{\d t}( \phi \circ \gamma)(t) \,\d t
    = \int_{\Omega} \psi \, \d(\overline{\nabla \phi \cdot \llbracket\gamma\rrbracket}).
    \end{split}
  \end{equation} 
  By taking $\psi = \mathbbm{1}_E$ where $E \subset \Omega$ is any Borel set, we infer \ref{item:curve_measure_conv}.
\end{proof}

\begin{proof}[Proof of Theorem \ref{thm:pairing_pointwise}]
  Given $\phi \in \WW^{1,\infty}(\Omega)$, by Lemma~\ref{lem:w1infty_conv} there exists $(\phi_k)_k \subset \CC_{\rmb}^1(\Omega)$ converging weakly${}^{\ast}$ to $\phi$ as $k \to \infty$, so $\phi_k \to \phi$ pointwise in $\Omega$ and $M := \sup_k \lVert \phi_k \rVert_{\WW^{1,\infty}(\Omega)} < \infty$.
  Using the decomposition of Theorem~\ref{thm:decomposition_domain}, we obtain a measure $\nu$ on $\mathscr{C}_1$ such that \eqref{eq:decomposition_domain} holds, so for $\psi \in \mathcal B_{\rmb}(\Omega)$ and each $k$,
  \begin{equation}
    \int_{\Omega} \psi \nabla \phi_k \cdot \d \FF = \int_{\mathscr{C}_1} \langle \llbracket\gamma\rrbracket, \psi \nabla \phi_k \rangle  \,\d \nu(\gamma).
  \end{equation} 
  Noting that $\Gamma_{\gamma} \subset \Omega$ for $\nu$-a.e.\,$\gamma\in\mathscr{C}_1$ (and the null set is Borel measurable), for any such $\gamma$ we have by Lemma~\ref{lem:pairing_curve}\ref{item:curve_measure_rep}, \ref{item:curve_measure_conv} that
  \begin{equation}
    \lim_{k \to \infty} \langle \llbracket\gamma\rrbracket , \psi \nabla \phi_k  \rangle = \int_{\Gamma_{\gamma}} \psi \,\d(\overline{\nabla\phi\cdot\llbracket \gamma \rrbracket}) = \int_{\Gamma_{\gamma}} \psi \nabla_{\llbracket\gamma\rrbracket}\phi \cdot \xi_{\gamma} \,\d \mathcal H^1.
  \end{equation} 
  Moreover by Lemma~\ref{lem:directional_diff}, for $\nu$-a.e.\,$\gamma \in \mathscr C_1$, $\nabla_{\llbracket\gamma\rrbracket}\phi\cdot\xi_{\gamma} = \nabla_{\FF}\phi\cdot\xi_{\gamma}$ holds $\mathcal H^1$-a.e.\,on $\Gamma_{\gamma}$.
  Hence there is a Borel measurable $\nu$-null set $\mathcal N \subset \mathscr C_1$ such that defining functions $(\Psi_k)_k, \Psi$ on $\mathscr{C}_1$ by
  \begin{equation}
    \Psi_k(\gamma) = \mathbbm{1}_{\mathcal N}(\gamma) \langle \llbracket\gamma\rrbracket,\psi\nabla\phi_k\rangle, \quad \Psi(\gamma) = \mathbbm{1}_{\mathcal N} \int_{\Gamma_{\gamma}} \psi \nabla_{\FF}\phi \cdot \xi_{\gamma} \,\d\mathcal H^1,
  \end{equation} 
  we have each $\Psi_k$ is Borel measurable and that $\Psi_k(\gamma) \to \Psi(\gamma)$ for all $\gamma \in \mathscr{C}_1$, thereby implying the measurability of $\Psi$.
  Moreover for each $\gamma \in \mathscr C_1 \setminus \mathcal N$ we can bound
  \begin{equation}
    \lvert \Psi_k(\gamma)\rvert = \left\lvert \int_{\Gamma_{\gamma}} \psi \nabla \phi_k \cdot \llbracket\gamma\rrbracket\right\rvert \leq \big(\sup_{\Omega}\lvert\psi\rvert\big) M\,\ell(\gamma) \quad\mbox{for all $k$},
  \end{equation} 
  and the same bound holds on $\mathcal N$ since each $\Psi_k$ vanishes there.
  By Lemma \ref{lem:product_rule} and applying the dominated convergence theorem using \eqref{eq:decomposition_domain_tv_2}, we obtain
  \begin{equation}
    \begin{split}
      \int_{\Omega} \psi\,\d(\overline{\nabla\phi \cdot \FF}) &= \lim_{k \to \infty} \int_{\Omega} \psi \nabla \phi_k \cdot \d \FF
      = \lim_{k \to \infty} \int_{\mathscr C_1} \Psi_k(\gamma) \,\d\nu(\gamma) \\
      &\quad= \int_{\mathscr C_1} \Psi(\gamma) \,\d\nu(\gamma)= \int_{\mathscr{C}_1} \int_{\Gamma_{\gamma}} \psi \nabla_{\FF}\phi \cdot \xi_{\gamma}\,\d\mathcal H^1 \,\d\nu(\gamma),
    \end{split}
  \end{equation} 
  establishing \eqref{eq:pairing_disint}.
  Therefore by combining the above with \eqref{eq:decomposition_domain} and using the measurability of $\nabla_{\FF}\phi$ from Lemma~\ref{lem:directional_diff}, it follows that
  \begin{equation}
    \int_{\Omega} \psi \,\d(\overline{\nabla \phi \cdot \FF}) = \int_{\mathscr C_1} \langle \llbracket\gamma\rrbracket,\psi \nabla_{\FF}\phi \rangle\,\d\nu(\gamma) = \int_{\Omega} \psi \nabla_{\FF}\phi \cdot \d \FF
  \end{equation} 
  holds.
  That is,
  $\overline{\nabla \cdot \FF} = \nabla_{\FF}\phi \cdot \FF$
  as measures in $\Omega$, from which we infer \eqref{eq:precise_product} by uniqueness of the Radon-Nikod\'ym decomposition (see \emph{e.g.}\,\cite[Thm.\,1.28]{AmbrosioEtAl2000}).
\end{proof}

\begin{proof}[Proof of Theorem \ref{thm:setwise_conv}]
  Let $\phi_k \weaksto \phi$ as in the statement and $E \subset \Omega$ be any Borel set.
  Decomposing $\FF$ using Theorem~\ref{thm:decomposition_domain},
  by \eqref{eq:pairing_disint} from Theorem~\ref{thm:pairing_pointwise} we can write
  \begin{equation}\label{eq:setwise_decomp}
    \int_{E} \d(\overline{\nabla\phi_k \cdot \FF}) = \int_{\mathscr C_1} \int_{\Gamma_{\gamma} \cap E}  \d(\overline{\nabla\phi_k \cdot \llbracket\gamma\rrbracket}) \,\d \nu(\gamma).
  \end{equation} 
  for each $k$.
  Then for each $\gamma \in \mathscr{C}_1$ such that $\Gamma_{\gamma} \subset \Omega$, which holds for $\nu$-a.e.\,$\gamma\in\mathscr{C}_1$, by Lemma~\ref{lem:pairing_curve}\ref{item:curve_measure_conv} we have
  \begin{equation}
    \lim_{k \to \infty} \int_{\Gamma_{\gamma} \cap E}  \d(\overline{\nabla\phi_k \cdot \llbracket\gamma\rrbracket}) = \int_{\Gamma_{\gamma} \cap E}  \d(\overline{\nabla\phi \cdot \llbracket\gamma\rrbracket}).
  \end{equation} 
  Also by weak${}^{\ast}$ convergence of $\phi_k$, we have $M = \sup_k \lVert \phi_k \rVert_{\WW^{1,\infty}(\Omega)} < \infty$, so using this we obtain the uniform bound
  \begin{equation}
    \left\lvert\int_{\Gamma_{\gamma} \cap E}  \d(\overline{\nabla\phi_k \cdot \llbracket\gamma\rrbracket})\right\rvert \leq \lVert\nabla\phi_k \rVert_{\LL^{\infty}(\Omega)}\, \mu_{\llbracket\gamma\rrbracket}(\Gamma_{\gamma}) \leq M \ell(\gamma),
  \end{equation} 
  for all $\gamma \in \mathscr{C}_1$ such that $\Gamma_{\gamma} \subset \Omega$.
  Since the right-hand side is $\nu$-integrable over $\gamma \in \mathscr{C}_1$ by \eqref{eq:decomposition_domain_tv_2},
  by the dominated convergence theorem we can send $k \to \infty$ in \eqref{eq:setwise_decomp} to infer \eqref{eq:setwise_claim}.
\end{proof}

\subsection{Applications to the normal trace}\label{sec:pairing_applications}

Equipped with Theorems~\ref{thm:pairing_pointwise} and \ref{thm:setwise_conv}, we can prove the following result claimed in the introduction.

\begin{thm}\label{lem:optimal_trace}
  Let $\Omega \subset \mathbb R^n$ be an open set, then for any Borel set $E \subset \Omega$ there exists a unique linear mapping
  \begin{equation}
    \mathrm N_E \colon \mathcal{DM}^{\ext}(\Omega) \to \text{\AE}(\partial E) \subset \Lip(\partial E)^{\ast}
  \end{equation} 
  which for each $\FF \in \mathcal{DM}^{\ext}(\Omega)$ satisfies
  \begin{equation}\label{eq:optimal_tracedefinition}
    \langle\mathrm{N}_E(\FF),\phi \rvert_{\partial E}\rangle_{\text{\AE}} = -\int_E \nabla \phi \cdot \FF - \int_E \phi\,\d(\div \FF) \quad\mbox{for all } \phi \in \CC^1_{\rmb}(\Omega), 
  \end{equation} 
  and is bounded in that
  $\lVert \mathrm{N}_E(\FF) \rVert_{\text{\AE}(\partial E)} \leq C \lVert \FF \rVert_{\mathcal{DM}^{\ext}(\Omega)}$, where $C = C(n)$.
\end{thm}

This will largely follow from two results of independent interest.

\begin{prop}\label{prop:improved_support}
  Let $\FF \in \mathcal{DM}^{\ext}(\Omega)$ and $E \subset \Omega$ be a Borel set. If $\phi \in \Lip_{\rmb}(\Omega)$ vanishes on $\partial E$, then
  \begin{equation}
    \langle \FF \cdot \nu, \phi \rangle_{\partial E} = 0.
  \end{equation} 
\end{prop}

This improves \cite[Lem.\,3.2]{Silhavy2009} and \cite[Thm.\,2.15]{ChenEtAl2024}, where one additionally imposes a topological or measure-theoretic condition on $E$. 
Here, since any $\phi \in \Lip_{\rmb}(\Omega)$ admits a unique continuous extension to $\overline\Omega$, the condition $\phi \rvert_{\partial E} = 0$ can be understood even if $\partial E \not\subset \Omega$.

\begin{proof}
  By applying \cite[Thm.\,10.2]{ChenEtAl2024} to the interior $E^{\circ} = E \setminus \partial E$ of $E$ and noting that $\phi$ vanishes on $\partial E^{\circ} \subset \partial E$, we obtain
  \begin{equation}
    0 = \langle \FF \cdot \nu, \phi \rangle_{\partial E^{\circ}} = - \int_{E^{\circ}} \,\d(\overline{\nabla\phi \cdot \FF})  - \int_{E^{\circ}} \phi \,\d (\div \FF).
  \end{equation} 
  Hence we infer that
  \begin{equation}\label{eq:stp_is_zero}
    \langle \FF \cdot \nu, \phi \rangle_{\partial E} = -\int_{E} \phi \,\d(\div \FF) - \int_E \d(\overline{\nabla \phi \cdot \FF}) = - \int_{E \cap \partial E} \d(\overline{\nabla \phi \cdot \FF}),
  \end{equation} 
  so the assertion would follow if the last integral vanishes.
  Using Theorem~\ref{thm:pairing_pointwise} we can write 
  \begin{equation}\label{eq:partial_E_decomp}
    \int_{E \cap \partial E} \d(\overline{\nabla \phi \cdot \FF}) = \int_{\mathscr{C}_1} \int_{E \cap \partial E \cap \Gamma_{\gamma}} \nabla_{\FF} \phi \cdot \xi_{\gamma}\,\d \mathcal H^1 \,\d\nu(\gamma),
  \end{equation} 
  where $\nu$ is the measure obtained in the decomposition of Theorem~\ref{thm:decomposition_domain}.
  Now for any $\gamma \in \mathscr{C}_1$ such that $\Gamma_{\gamma} \subset \Omega$, by the area formula we have
  \begin{equation}\label{eq:partial_E_gamma}
    \int_{E \cap \partial E \cap \Gamma_{\gamma}} \nabla_{\FF} \phi \cdot \xi_{\gamma}\,\d \mathcal H^1  = \int_0^1 \mathbbm{1}_{A(\gamma,E)}(t) (\phi \circ \gamma)'(t) \,\d t,
  \end{equation} 
  where
  \begin{equation}
    A(\gamma,E) = \{t \in [0,1] \colon \phi(\gamma(t)) \in E \cap \partial E \}.
  \end{equation} 
  However since $\phi \circ \gamma \equiv 0$ on $A(\gamma,E)$, applying \cite[Thm.\,3.3]{EvansGariepy2015} it holds that $(\phi \circ \gamma)' = 0$ $\mathcal L^1$-a.e.\,on $A(\gamma,E)$.
  Therefore \eqref{eq:partial_E_gamma} vanishes for any $\gamma \in \mathscr{C}_1$ contained in $\Omega$, which holds for $\nu$-a.e.\,$\gamma \in \mathscr C_1$. 
  Therefore combining this with \eqref{eq:stp_is_zero}, \eqref{eq:partial_E_decomp} we obtain
  \begin{equation}
    \langle \FF \cdot \nu, \phi \rangle_{\partial E} = -\int_{E \cap \partial E} \,\d(\overline{\nabla\phi\cdot\FF}) = -\int_{\mathscr{C}_1} \int_0^1 \mathbbm{1}_{A(\gamma,E)}(t)(\phi \circ \gamma)'(t) \,\d t \,\d\nu(\gamma) = 0,
  \end{equation} 
  as required.
\end{proof}

\begin{prop}\label{prop:normaltrace_weakcont}
  Let $\FF \in \mathcal{DM}^{\ext}(\Omega)$ and $E \subset \Omega$ be a Borel set.
  Then the normal trace $\langle \FF \cdot \nu,\,\cdot\,\rangle_{\partial E}$ extends uniquely from $\CC^1_{\rmb}(\Omega)$ to a weakly${}^{\ast}$-continuous functional on $\WW^{1,\infty}(\Omega)$.
\end{prop}
\begin{proof}
  Since $\CC^1_{\rmb}(\Omega)$ is weakly${}^{\ast}$ dense in $\WW^{1,\infty}(\Omega)$ by Lemma~\ref{lem:w1infty_conv}, it suffices to show that
  \begin{equation}
    \langle \FF \cdot \nu,\phi \rangle_{\partial E} = -\int_{E} \d(\overline{\nabla\phi\cdot\FF}) - \int_{E} \phi \,\d(\div \FF)
  \end{equation} 
  is weakly${}^{\ast}$-continuous.
  For this let $(\phi_k)_k \subset \WW^{1,\infty}(\Omega)$ such that $\phi_k \weaksto \phi$ weakly${}^{\ast}$ in $\WW^{1,\infty}(\Omega)$, then we know that $(\phi_k)_k$ is uniformly bounded in $\WW^{1,\infty}(U)$ and $\phi_k \to \phi$ pointwise in $\Omega$ by Lemma~\ref{lem:w1infty_conv}.
  Then $\int_E \phi_k\,\d(\div \FF) \to \int_E \phi\,\d(\div \FF)$ as $k \to \infty$ by the dominated convergence theorem, and by the setwise convergence of Theorem~\ref{thm:setwise_conv}, we also have
  that $\overline{\nabla\phi_k\cdot\FF}(E) \to \overline{\nabla\phi \cdot \FF}(E)$.
  Hence it follows that
  \begin{equation}
    \lim_{k \to \infty} \langle \FF \cdot \nu,\phi_k \rangle_{\partial E} = \langle \FF \cdot \nu,\phi \rangle_{\partial E}.
  \end{equation} 
  This proves that $\langle \FF \cdot \nu,\,\cdot\,\rangle_{\partial E}$ is sequentially weakly${}^{\ast}$-continuous.
  Since $\WW^{1,\infty}(\Omega)$ has a separable predual by Lemma \ref{lem:w1infty_predual}, a consequence of the Krein-\v{S}mulian theorem, namely \cite[Cor.\,12.8]{Conway2007} implies that the normal trace is weakly${}^{\ast}$-continuous.
\end{proof}

\begin{proof}[Proof of Theorem \ref{lem:optimal_trace}]
  Let $U = B_1(\partial E) = \{ x \in \mathbb R^n : \dist(x,\partial E) < 1\}$, then applying \cite[Thm.\,2]{JohnsonEtAl1986} we obtain a linear extension operator $\Lip_{\rmb}(\partial E) \to \Lip_{\rmb}(U)$ which sends $\phi \to \tilde\phi$ and satisfies $\Lip(\tilde\phi,U) \leq C(n) \Lip(\phi,\partial U)$.
  By multiplying this extension with $\chi(x) = \max\{1 - \dist(x,\partial E),0\}$ which is $1$-Lipschitz and vanishes on $\partial U$, we obtain a bounded linear mapping $T \colon \Lip_{\rmb}(\partial E) \to \Lip_{\rmb}(\Omega)$ by sending $\phi \to \chi \tilde \phi$.
  Moreover, by the explicit construction of $\tilde\phi$ given in \cite[pp.\,133]{JohnsonEtAl1986}, we see this mapping is continuous with respect to uniform convergence on bounded sets, and hence by  Lemma~\ref{lem:lipb_conv}, $T$ is sequentially continuous with respect to the respective weak${}^{\ast}$-topologies.

  Now given $\FF \in \mathcal{DM}^{\ext}(\Omega)$, we will define $\mathrm{N}_E(\FF) \in \Lip_{\rmb}(\partial E)^{\ast}$ by setting
  \begin{equation}\label{eq:trace_definition}
    \langle \mathrm{N}_E(\FF), \phi \rangle := \langle \FF \cdot \nu, T\phi \rangle_{\partial E} \quad\mbox{for each $\phi \in \Lip_{\rmb}(\partial E)$}.
  \end{equation} 
  By Proposition~\ref{prop:normaltrace_weakcont}, the normal trace $\langle \FF \cdot \nu, \,\cdot\,\rangle_{\partial E}$ is weakly${}^{\ast}$-continuous on $\Lip_{\rmb}(\Omega)$, so combined with the sequential continuity of $T$, it follows that $\mathrm{N}_E(\FF)$ is a sequentially weakly${}^{\ast}$-continuous linear functional on $\Lip_{\rmb}(\partial E)$.
  Using the identification $\Lip_{\rmb}(\partial E) \cong \text{\AE}(\partial E)^{\ast}$ induced by the pairing $\langle \cdot,\cdot\rangle_{\text{\AE}}$ from Proposition~\ref{prop:ae_isometric}, we can view $\mathrm N_E(\FF)$ as a sequentially weakly${}^{\ast}$-continuous linear functional on $\text{\AE}(\partial E)^{\ast}$. 
  Furthermore since $\text{\AE}(\partial E)$ is separable, by \cite[Cor.\,12.8]{Conway2007} we have $\mathrm{N}_E(\FF)$ is weakly${}^{\ast}$-continuous,
  and hence there exists $m \in \text{\AE}(\partial E)$ such that
  \begin{equation}
    \langle\mathrm N_U(\FF),\phi\rangle = \langle m, \phi \rangle_{\text{\AE}} \quad\mbox{for all $\phi \in \Lip_{\rmb}(\partial E)$}.
  \end{equation} 
  Therefore $\mathrm N_U(\FF) = m \in \text{\AE}(\partial E) \subset \Lip_{\rmb}(\partial E)^{\ast}$.

  Finally by Proposition~\ref{prop:improved_support}, $\mathrm{N}_E$ as defined in \eqref{eq:trace_definition} does not depend on the particular choice of extension; that is, for any $\tilde\phi \in \Lip_{\rmb}(\Omega)$ such that $\tilde\phi\rvert_{\partial E} = \phi$, we have $\langle \mathrm{N}_E(\FF),\phi \rangle = \langle \FF \cdot \nu, \tilde\phi \rangle_{\partial E}$.
  Therefore it follows that \eqref{eq:optimal_tracedefinition} holds, and since $\CC^1_{\rmb}(\Omega)$ is weakly${}^{\ast}$-dense in $\WW^{1,\infty}(\Omega)$ by Lemma~\ref{lem:w1infty_conv}, this uniquely determines $\mathrm{N}_E$.
\end{proof}

\begin{remark}
  Since Proposition~\ref{prop:normaltrace_weakcont} asserts that $\langle \FF \cdot \nu, \,\cdot\,\rangle_{\partial E}$ is weakly${}^{\ast}$ continuous in $\WW^{1,\infty}(U)^{\ast}$, by a similar argument to that in the proof of Theorem~\ref{lem:optimal_trace}, we can view the normal trace as an element of $\WW^{-1,1}(U)$. 
  This observation will have consequences for $\mathcal{DM}^1$-fields, which will be explored in \S\ref{sec:dm1}.
\end{remark}

\section{Image of the normal trace}

\subsection{Surjectivity of the trace operator}\label{sec:trace_surj}

In this section we will prove Theorem \ref{thm:intro_surjective}; namely that under a mild regularity condition, the normal trace operator $\mathrm{N}_E$ constructed in Theorem~\ref{lem:optimal_trace} is in fact surjective.
We will give a more precise statement in Theorem \ref{thm:surjective_trace}, however to specify the location of singularities and the dependence of constants, we will first define some relevant quantities.
We begin by recalling the following definition from the introduction.

\begin{definition}\label{defn:domain_lrc}
  We say an open set $U \subset \mathbb R^n$ is \emph{locally uniformly quasiconvex} if there exists $\eps, \delta>0$ such that for all $p,q \in \overline U$ such that $\lvert p - q\rvert \leq \delta,$ there exists a rectifiable curve $\gamma \colon [0,1] \to \mathbb R^n$ connecting $p$ to $q$ via $U$ in that
  $\gamma(0)=p$, $\gamma(1) = q$ and $\gamma(t) \in U$ for all $t \in (0,1)$
  that satisfies the estimate
  \begin{equation}\label{eq:quasiconvex_estimate}
    \ell(\gamma) \leq \eps^{-1}\lvert p - q\rvert.
  \end{equation} 
  Sometimes we will denote these constants by $\eps_U, \delta_U$ to specify the underlying open set.
\end{definition}

Our definition is slightly non-standard, as we impose this connectivity condition to hold up to the boundary (compare with \emph{e.g.}\,\cite[\S 2.1]{Heinonen2005}).
However, the following lemma asserts the two notions are, up to modifying constants, equivalent.

\begin{lem}\label{lem:domain_boundary_qc}
  Suppose $U \subset \mathbb R^n$ is an open subset satisfying the condition of Definition~\ref{defn:domain_lrc} the interior; that is, there exists $\eps,\delta>0$ such that for all $p,q \in U$ with $\lvert p -q\rvert< \delta$, there is a rectifiable curve connecting $p$ to $q$ in $U$ such that \eqref{eq:quasiconvex_estimate} holds.
  Then $U$ is locally uniformly quasiconvex in the sense of Definition~\ref{defn:domain_lrc} with constants $(\eps/2,\delta/2)$.
\end{lem}
\begin{proof}
  Let $p,q \in \overline U$ such that $D := \lvert p - q\rvert \leq \frac12\delta$. We will choose sequences $(p_k)$, $(q_k)$ in $U$ such that
  \begin{equation}
    \lvert p - p_k \rvert \leq 2^{-(k+3)} D, \quad \lvert q - q_k \rvert \leq 2^{-(k+3)} D \quad\mbox{for all $k \in \mathbb N$}.
  \end{equation} 
  Then since $\lvert p_1 - q_1\rvert \leq \frac{3}2 D \leq \delta$, there exists a rectifiable curve $\gamma_1$ connecting $p_1$ to $q_1$ in $U$ such that $\ell(\gamma_1) \leq \frac{3}2 \eps^{-1} D$.
  We can similarly find rectifiable curves $(\alpha_k)$, $(\beta_k)$ in $U$ such that
  $\alpha_k$ connects $p_{k+1}$ to $p_k$ such that $\ell(\alpha_k) \leq 2^{-(k+2)}\eps^{-1}D$,
  and $\beta_k$ connects $q_{k}$ to $q_{k+1}$ such that $\ell(\beta_k) \leq 2^{-(k+2)}\eps^{-1}D$.
  Then we can define $\gamma$ as the concatenation of the curves $(\alpha_k), \gamma_1, (\beta_k)$; this is the curve represented by the current
  \begin{equation}
    \llbracket \gamma \rrbracket = \llbracket \gamma_1 \rrbracket + \sum_{k=1}^{\infty} (\llbracket \alpha_k\rrbracket + \llbracket \beta_k\rrbracket).
  \end{equation} 
  Then $\gamma$ connects $p$ to $q$ via $U$ and satisfies
  \begin{equation}
    \ell(\gamma) = \ell(\gamma_1) + \sum_{k=1}^{\infty} (\ell(\alpha_k)+\ell(\beta_k)) \leq \frac{2}{\eps} \lvert x- y\rvert,
  \end{equation} 
  as required.
\end{proof}

For some of our results we will impose that both $U$ and $\overline{U}^{\rmc}$ satisfies Definition~\ref{defn:domain_lrc} and that $\partial U = \partial \overline{U}^{\rmc}$; this is satisfied by all bounded Lipschitz domains, but is more general as the below example illustrates.

\begin{eg}\label{eg:koch_snowflake}
  The standard Koch snowflake $S \subset \mathbb R^2$ has the property that both $S$ and $\overline{S}^{\rmc}$ satisfies Definition~\ref{defn:domain_lrc} and that $\partial S = \partial \overline S^{\rmc}$, despite having fractal boundary.
  Indeed it was observed for instance in \cite{Jones1981} that $S$ is a $(\eps,\delta)$-domain, which provides the desired connectivity in the interior.
  For the exterior, we note the well-known property that $S$ may be tiled by rescaled copies of itself, so a neighbourhood of $\overline S^{\rmc}$ can be expressed as the union of finitely many uniform domains.
\end{eg}

We will record some basic properties concerning the connected components of $U$.
In what follows, we will assume Definition~\ref{defn:domain_lrc} holds for all $p,q \in \overline U$ without chasing the constant; 

\begin{lem}\label{lem:domain_cc}
  Let $U \subset \mathbb R^n$ be an open set satisfying Definition~\ref{defn:domain_lrc}, then the following holds:
  \begin{enumerate}[label=(\alph*)]
    \item\label{item:domain_cc} If $C_1, C_2$ are distinct connected components of $U$, then $\dist(C_1,C_2) \geq \delta$.

    \item\label{item:domain_bdry} For each $p \in \partial U$, there exists a unique connected component $C(p)$ of $U$ such that $p \in \partial C(p)$.

    \item\label{item:domain_connect} For each $p \in \partial U$ and $q \in C(p)$, there exists a rectifiable curve $\gamma \colon [0,1] \to \mathbb R^n$ such that $\gamma(0) = p$, $\gamma(1) = q$ and $\gamma(t) \in U$ for all $t \in (0,1)$.
  \end{enumerate}
\end{lem}

\begin{proof}
  For \ref{item:domain_cc}, let $C_1, C_2$ be distinct connected components of $U$. Then any two points $p_1 \in C_1$ and $p_2 \in C_2$ cannot be connected within $U$ and hence must satisfy $\lvert p_1 - p_2\rvert \geq \delta$, from which we infer that $\dist(C_1,C_2) \geq \delta$.

  Now let $p \in \partial U$, then $p$ can be connected in $U$ to every point $q \in B_{\delta}(p) \cap U$, which is non-empty since $p \in \partial U$.
  Hence there is a unique connected component $C(p)$ containing $B_{\delta}(p) \cap U$, so $p \in \partial C(p)$ and hence \ref{item:domain_bdry} holds.
  Also since $C(p)$ is open and connected in $\mathbb R^n$, it is path-connected.
  Hence for any $q \in C(p)$, choosing $\tilde q \in B_{\delta}(p) \cap U$ we can connect $p$ to $\tilde q$ and $\tilde q$ to $q$, thereby establishing \ref{item:domain_connect}.
\end{proof}

\begin{lem}\label{lem:concentration_set}
  Let $U \subset \mathbb R^n$ be an open set satisfying Definition~\ref{defn:domain_lrc}.
  Then there exists a discrete set $\Lambda \subset U$ such that $\dist(\Lambda,\partial U) > 0$, and the \emph{separation}
  \begin{align}
    \label{eq:geodist_in}\sep_{\gamma}(\Lambda,U) &:= \sup_{p \in \partial U} \inf\{\ell(\gamma) :\gamma \in \Lip_{\rmb}((0,1);U),\ \gamma(0) = p,\ \gamma(1)\in\Lambda\} \quad\mbox{is finite}.
  \end{align} 
  Moreover, if $\partial U$ is bounded, we can take $\Lambda$ to be any finite set containing at least one point from each connected component of $U$.
\end{lem}

\begin{proof}
  For general $U$, let $\Lambda = \{ p_k \}_k$ be an at most countable set of points such that $\dist(p_k,\partial U) \geq \delta/2$ for each $k$ and such that $\{ B_{\delta}(p_k)\}_k$ covers $U$.
  For any $q \in \partial U$, since there is some $p_k \in \Lambda$ such that $\lvert q  - p_k\rvert < \delta$, there exists a curve $\gamma$ connecting $q$ to $p_k$ in $U$ such that $\ell(\gamma) \leq \eps^{-1}\delta$. This shows that $\sep_{\gamma}(\Lambda,U) \leq \eps^{-1}\delta$, and since $\dist(\Lambda,\partial U) \geq \delta/2$ by construction, $\Lambda$ satisfies the claimed properties.

  If $\partial U$ is bounded, then $\mathbb R^n \setminus \partial U$ has at most one unbounded component.
  Therefore since the connected components of $U$ are $\delta$-separated by Lemma~\ref{lem:domain_cc}\ref{item:domain_cc}, it follows that only finitely many such components exist.
  Hence, choosing $\Lambda \subset U$ to be a finite set containing at least one point from each connected component of $U$, by Lemma~\ref{lem:domain_cc}\ref{item:domain_connect} any $p \in \overline U$ can be connected to $\Lambda$ via $U$.
  Now if we define $\partial U \ni p \mapsto \inf \ell(\gamma)$, where the infimum is taken as in \eqref{eq:geodist_in}, then this mapping is well-defined and lower semicontinuous. 
  Thus by compactness of $\partial U$ we infer that $\mathrm{sep}_{\gamma}(\Lambda,U) < \infty$,
  and since $\partial U$ and $\Lambda$ are disjoint and compact we have $\dist(\Lambda,\partial U)>0$.
\end{proof}

\begin{lem}\label{lem:elementary_surject}
  Let $U \subset \mathbb R^n$ be an open set satisfying Definition~\ref{defn:domain_lrc}, and let $\Lambda \subset U$ be a discrete set satisfying the conclusions of Lemma~\ref{lem:concentration_set}.
  Then for each $m \in \text{\AE}_0(\partial U)$, there exists $\FF \in \mathcal{DM}^{\ext}(U)$ such that $\div \FF$ is supported in $\Lambda$, 
  \begin{equation}\label{eq:preextension_trace}
    (\FF \cdot \nu)_{\partial U} =m \quad\mbox{in $\text{\AE}(\partial U)$},
  \end{equation} 
  and the estimate
  \begin{equation}\label{eq:preextension_estimate}
    \lVert \FF \rVert_{\mathcal{DM}^{\ext}(U)} \leq C \lVert m \rVert_{\text{\AE}(\partial U)}
  \end{equation} 
  holds, where $C = C(n,\eps_U,\delta_U,\sep_{\gamma}(\Lambda,U),\dist(\Lambda,\partial U))$. 
  We denote the associated mapping $m \mapsto \FF$ by $\mathrm E_U \colon \text{\AE}_0(\partial U) \to \mathcal{DM}^{\ext}(U)$.
\end{lem}

\begin{proof}
  We will fix any point $e \in \Lambda \subset U$, and shrinking $\delta$ if necessary we can assume that $\delta < \min\{1,\dist(e,\partial U)\}$.
  Let $m \in \text{\AE}_0(\partial U)$, then using the identification $\text{\AE}_0(\partial U) \cong A_0(\partial U,e)$ from Proposition~\ref{prop:ae_isometric}\ref{item:ae_description} there exists a representation
  \begin{equation}
    \tilde m := m - m(\partial U) \delta_e = \sum_{i=1}^k a_i (\delta_{q_i}-\delta_{p_i}) \in A(\partial U, e),
  \end{equation} 
  where $k \in \mathbb N$ and each $a_i \in \mathbb R$, $p_i, q_i \in \partial U \cup\{e\}$ such that
  \begin{equation}\label{eq:m_norm_below}
    \sum_{i=1}^k \lvert a_i \rvert \rho(p_i,q_i)\rvert = \lVert m \rVert_{\text{\AE}(\partial U)},
  \end{equation} 
  where $\rho$ is the metric defined by \eqref{eq:metric_rho}, \eqref{eq:metric_rho2}.

  \textbf{Claim}: For each $1 \leq i \leq k$, there exists $\FF_i \in \mathcal{DM}^{\ext}(U)$ associated to $p_i$, $q_i$ such that:
  \begin{enumerate}[label=(\roman*)]
    \item\label{item:claim_curve} $\FF_i$ is represented by curves contained in $U$,

    \item\label{item:claim_div_supp} $\div \FF_i$ in $U$ is supported on $\Lambda$, and vanishes if $p_i, q_i \in \partial U$ and $\lvert p_i - q_i\rvert \leq \delta$.

    \item\label{item:claim_length} $\lvert \FF_i\rvert(U)  \leq C \rho(p_i, q_i)$, where the $C$-dependence as in the statement,
    \item\label{item:claim_trace} the normal trace of $\FF_i$ on $\partial U$ is given by
      \begin{align}
        (\FF_i \cdot \nu)_{\partial U}&= (\delta_{q_i} - \delta_{p_i}) \mres \partial U.
      \end{align}

  \end{enumerate}
  For this we will distinguish between the following three cases.

  \noindent\emph{Case 1}: Suppose $p_i,q_i \in \partial U$ with $\lvert q_i - p_i\rvert \leq \delta$.
  In this case, by Definition~\ref{defn:domain_lrc} there exists a curve $\alpha_i$ connecting $p_i$ to $q_i$ through $U$, such that $\ell(\alpha_i) \leq \eps^{-1} \lvert q_i-p_i\rvert = \eps^{-1} \rho(p_i,q_i)$ (noting that $\delta \leq 1$).
  We then let $\FF_i = \llbracket \alpha_i\rrbracket$ and observe this satisfies \ref{item:claim_curve}--\ref{item:claim_length}, since $\div \FF_i = \delta_{q_i} - \delta_{p_i}$ in $\mathbb R^n$, which vanishes in $U$.
  This also implies that $(\FF_i \cdot \nu)_{\partial U} = \delta_{q_i} - \delta_{p_i}$ using \eqref{eq:normaltrace_distributional}, establishing \ref{item:claim_trace}.

  \noindent\emph{Case 2}:\linkdest{case2} Suppose $p_i,q_i \in \partial U$ with $\lvert q_i - p_i\rvert > \delta$.
  We let $\alpha_i, \beta_i$ be curves connecting $p_i$ and $q_i$ to points in $\Lambda$ respectively, such that $\ell(\alpha_i), \ell(\beta_i) \leq 2 \sep_{\gamma}(\Lambda,U)$,
  which exists since the separation \eqref{eq:geodist_in} is finite.
  Then we can set $\FF_i = \llbracket \alpha_i \rrbracket - \llbracket \beta_i \rrbracket$,
  which evidently satisfies \ref{item:claim_curve} and \ref{item:claim_div_supp}, and \ref{item:claim_length} holds since
  \begin{equation}\label{eq:case2_finfout_estimate}
  \lvert \FF_i \rvert(\mathbb R^n) \leq 4\sep_{\gamma}(\Lambda,U) \leq 4\delta^{-1}\sep_{\gamma}(\Lambda,U)\, \rho(p_i,q_i).
  \end{equation} 
  For \ref{item:claim_trace} we note that $(\llbracket \alpha_i\rrbracket \cdot \nu)_{\partial U} = -\delta_{p_i}$ and $(\llbracket \beta_i\rrbracket \cdot \nu)_{\partial U} = -\delta_{q_i}$.

  \noindent\emph{Case 3}: Suppose that either $p_i = e$ or $q_i=e$; without loss of generality we can assume that $q_i=e$, and also that $p_i \in \partial U$ since otherwise $a_i(\delta_{q_i} - \delta_{p_i})$ would be zero.
  Similarly as in Case~\hyperlink{case2}{2}, we choose a curve $\alpha_i$ connecting $p_i$ to a point in $\Lambda$ with $\ell(\alpha_i) \leq 2 \sep_{\gamma}(\Lambda,U)$.
  We then set $\FF_i = \llbracket \alpha_i \rrbracket$, which since $\rho(p_i,q_i)=1$ satisfies
  \begin{equation}
    \lvert \FF_i\rvert(U) \leq \frac{2 \sep_{\gamma}(\Lambda,U)}{\dist(\Lambda,\partial U)} \rho(p_i,q_i),
  \end{equation} 
  From this \ref{item:claim_curve}--\ref{item:claim_trace} follows, noting as in Case~\hyperlink{case2}{2} that $(\llbracket \alpha_i \rrbracket \cdot \nu)_{\partial U} = -\delta_{p_i}$.

  This establishes the claim in every case, so we can now define
  \begin{equation}
    \FF = \sum_{i=1}^k a_i  \FF_i.
  \end{equation} 
  Then $\div \FF$ is supported on $\Lambda$ by property \ref{item:claim_div_supp} of the claim. 
  Now let $I \subset \{1,\cdots,k\}$ denote the indices $i$ for which $\div \FF_i \neq 0$ in $U$, which by property \ref{item:claim_div_supp} and since $\dist(e,\partial U) > \delta$, occurs if and only if $\rho(p_i,q_i) \geq \delta$.
  Using this we can estimate
  \begin{equation}
    \lvert \div \FF\rvert(U) \leq \sum_{i \in I} 2\lvert a_i\rvert \leq \sum_{i \in I} 2\lvert a_i\rvert \frac{\rho(p_i,q_i)}{\delta} \leq \frac{2}{\delta} \lVert m \rVert_{\text{\AE}(\partial U)}.
  \end{equation} 
  Also by property \ref{item:claim_length} and \eqref{eq:m_norm_below},
  \begin{equation}
    \lvert \FF\rvert(U) \leq C\sum_{i=1}^k \lvert a_i\rvert \rho(p_i,q_i) = C \lVert m \rVert_{\text{\AE}(\partial U)},
  \end{equation} 
  so $\FF$ satisfies the claimed estimate \eqref{eq:preextension_estimate}.
  Finally the trace property \eqref{eq:preextension_trace} follows from property \ref{item:claim_trace} of the claim along with the linearity of $\mathrm{N}_U$, thereby establishing the result.
\end{proof}

To establish Theorem \ref{thm:intro_surjective} from the introduction, it remains to extend this operator from $\text{\AE}_0(\partial U)$ to the completion $\text{\AE}(\partial U)$ by means of a non-linear density argument.

\begin{thm}\label{thm:surjective_trace}
  Let $U \subset \mathbb R^n$ be an open set that is locally uniformly quasiconvex in the sense of Definition~\ref{defn:domain_lrc}, and let $\Lambda \subset U$ satisfy the conclusion of Lemma~\ref{lem:concentration_set}.
  Then there exists a (not necessarily linear) mapping 
  \begin{equation}\label{eq:extension_leftinverse}
    \mathrm E_U \colon \text{\AE}(\partial U) \to \mathcal{DM}^{\ext}(U)
  \end{equation} 
  which is a left-inverse of the normal trace in that
  \begin{equation}\label{eq:extension_property}
    \mathrm N_U \circ \mathrm E_U = \Id_{\text{\AE}(\partial U)},
  \end{equation}
  and is bounded in that
  \begin{equation}\label{eq:extension_estimate2}
    \lVert \mathrm{E}_U(m) \rVert_{\mathcal{DM}^{\ext}(U)} \leq C \lVert m \rVert_{\text{\AE}(\partial U)}
    \quad\mbox{for all $m \in \text{\AE}(\partial U)$,}
  \end{equation} 
  where $C = C(n,\eps_U,\delta_U,\sep_{\gamma}(\Lambda,U),\dist(\Lambda,\partial U))$.
  Furthermore, each $\mathrm E_U(m)$ is divergence-free in $U \setminus \Lambda$.
\end{thm}

\begin{remark}\label{rem:choice_lambda}
We stress that $\Lambda$ merely needs to satisfy the conclusion of Lemma~\ref{lem:concentration_set}, but we can otherwise specify the position of the singularities; this will allow us to construct divergence-free extensions in \S\ref{sec:extension}.
However, if we choose $\Lambda$ to be as in the proof of Lemma~\ref{lem:concentration_set}, then $\sep_{\gamma}(\Lambda,U) \leq \eps^{-1}\delta$ and $\dist(\Lambda,\partial U) \geq \delta/2$, in which case the constant in \eqref{eq:extension_estimate2} depends on $n, \eps, \delta$ only.
\end{remark}

\begin{proof}
  Let $m \in \text{\AE}(\partial U)$, then by definition there exists a sequence $(m_k)_k \subset \text{\AE}_0(\partial U)$ such that $m_k \to m$ strongly in $\text{\AE}(\partial U)$.
  By passing to an unrelabelled subsequence we will assume that $\lVert m_1 \rVert_{\text{\AE}(\partial U)} \leq 2 \lVert m \rVert_{\text{\AE}(\partial U)}$ and that
  \begin{equation}
    \lVert m_{k} - m_{k-1} \rVert_{\text{\AE}(\partial U)} \leq 2^{-k}\lVert m \rVert_{\text{\AE}(\partial U)} \quad\mbox{for all $k \geq 2$}.
  \end{equation} 
  Using Lemma~\ref{lem:elementary_surject}, we then define
  \begin{equation}
    \FF_k = \sum_{j=1}^k \GG_j \quad\mbox{ where } \GG_k = 
    \begin{cases} 
      \mathrm{E}_U(m_1) & \text{if } k=1, \\
      \mathrm{E}_U(m_k-m_{k-1}) & \text{if } k \geq 2,
    \end{cases}
  \end{equation} 
  which we can estimate using \eqref{eq:preextension_estimate} as
  \begin{equation}\label{eq:absolute_convergence}
    \sum_{k=1}^{\infty} \lVert \GG_k \rVert_{\mathcal{DM}^{\ext}(U)} \leq C \left( \lVert m_1\rVert_{\text{\AE}(\partial U)}+\sum_{k=2}^{\infty} \lVert m_k - m_{k-1} \rVert_{\text{\AE}(\partial U)}\right) \leq C\lVert m \rVert_{\text{\AE}(\partial U)}.
  \end{equation} 
  Therefore the series
  \begin{equation}
    \FF = \sum_{k=1}^{\infty} \GG_k
  \end{equation} 
  converges absolutely in $\mathcal{DM}^{\ext}(U)$,
  so $\FF_k \to \FF$ strongly in $\mathcal{DM}^{\ext}(U)$ and the estimate
  \begin{equation}
    \lVert \FF \rVert_{\mathcal{DM}^{\ext}(U)} \leq C \lVert m \rVert_{\text{\AE}(\partial U)}
  \end{equation} 
  holds, where the dependence of constants is the same as in Lemma~\ref{lem:elementary_surject}.
  By \eqref{eq:preextension_trace} and using the linearity of the normal trace,
  \begin{equation}
    \mathrm{N}_U(\FF_k) = m_1 + \sum_{j=2}^k (m_{j} - m_{j-1}) = m_k
  \end{equation} 
  for all $k$, so
  using the strong convergence of $m_k$ and $\FF_k$, for any $\phi \in \CC^1_{\mathrm{b}}(U)$ we have
  \begin{equation}
    \begin{split}
      \langle \FF \cdot \nu, \phi \rangle_{\partial U} 
      &= -\int_{U} \nabla \phi \cdot \d\FF - \int_{U} \phi \,\d(\div \FF) \\
      &= \lim_{k \to \infty} \left( -\int_{U} \nabla \phi \cdot \d\FF_k - \int_{U} \phi \,\d(\div \FF_k) \right)  \\ 
      &= \lim_{k \to \infty} \langle m_k, \phi \rvert_{\partial U}\rangle_{\text{\AE}} 
      = \langle m, \phi\rvert_{\partial U} \rangle_{\text{\AE}}.
    \end{split}
  \end{equation} 
  Since this characterises the normal trace by Theorem~\ref{lem:optimal_trace}, it follows that
  $\mathrm{N}_U(\FF) = m$ in $\text{\AE}(\partial U)$.
  Therefore we can define $\mathrm{E}_U(m) = \FF$, which satisfies \eqref{eq:extension_property} and \eqref{eq:extension_estimate2} as claimed.
  Also since each $\div \FF_k$ is supported on the discrete set $\Lambda$ and $\div \FF_k \to \div \FF$ strongly in $\mathcal M(U)$, the same holds for $\div \FF$.
\end{proof}

Applying this to both $U$ and $\overline U^{\rmc}$, we also obtain the following two-sided variant.

\begin{thm}\label{cor:trace_surject_twosided}
  Let $U \subset \mathbb R^n$ be an open set such that both $U$, $\overline U^{\rmc}$ satisfies Definition~\ref{defn:domain_lrc} and that $\partial U = \partial \overline U^{\rmc}$.
  Additionally let $\Lambda \subset \mathbb R^n \setminus \partial U$ such that both $\Lambda \cap U$ and $\Lambda \setminus U$ satisfies the conclusion of Lemma~\ref{lem:concentration_set} in $U$ and $\overline U^{\rmc}$ respectively.
  Then there exists a (not necessarily linear) mapping 
  \begin{equation}
    \widetilde{\mathrm{E}}_U \colon \text{\AE}(\partial U) \to \mathcal{DM}^{\ext}(\mathbb R^n)
  \end{equation} 
  satisfying
  \begin{equation}
    \mathrm{N}_U \circ \widetilde{\mathrm{E}}_U = - \mathrm{N}_{\overline U^{\rmc}} \circ \widetilde{\mathrm{E}}_U = \Id_{\text{\AE}(\partial U)},
  \end{equation} 
  which is bounded in that
  \begin{equation}\label{eq:bounded_twosided}
    \lVert \widetilde{\mathrm{E}}_U(m) \rVert_{\mathcal{DM}^{\ext}(\mathbb R^n)} \leq C \lVert m \rVert_{\text{\AE}(\partial U)}
    \quad\mbox{for all $m \in \text{\AE}(\partial U)$.}
  \end{equation} 
  Moreover each $\FF = \widetilde{\mathrm{E}}_U(m)$ is divergence-free away from $\Lambda$ in $\mathbb R^n$, and satisfies $\lvert \FF\rvert(\partial U) = 0$.
\end{thm}

\begin{proof}
  Given the mappings $\mathrm{E}_U$ and $\mathrm{E}_{\overline U^{\rmc}}$ from Theorem~\ref{thm:surjective_trace}, for $m \in \text{\AE}(\partial U) = \text{\AE}(\partial \overline U^{\rmc})$ put
  \begin{equation}
    \FF_{\tin} := \mathrm{E}_U(m) \mres U, \quad \FF_{\tout} := \mathrm{E}_{\overline U^{\rmc}}(m) \mres \overline U^{\rmc},
  \end{equation} 
  which we view as measures on $\mathbb R^n$.
  Then we define $\FF = \widetilde{\mathrm{E}}_U(m) := \FF_{\tin} - \FF_{\tout}$, which satisfies $\lvert \FF\rvert(\partial U)= 0$ by construction.
  We claim $\FF$ satisfies the claimed properties; given $\phi \in \CC^1_{\rmb}(\mathbb R^n)$, we can compute the distributional divergence of $\FF$ as
  \begin{equation}
    \begin{split}
      \int_{\mathbb R^n} \nabla\phi \cdot \FF 
      &= \int_U \nabla \phi \cdot \FF_{\tin} - \int_{\overline U^{\rmc}} \nabla\phi \cdot \FF_{\tout} \\
      &= -\langle \FF_{\tin} \cdot \nu,\phi \rangle_{\partial U} - \int_U \phi \,\d(\div \FF_{\tin}) + \langle \FF_{\tout} \cdot \nu,\phi \rangle_{\partial \overline U^{\rmc}} + \int_{\overline U^{\rmc}} \phi\,\d(\div \FF_{\tout})\\
      &= - \left(\int_U \phi \,\d(\div \FF_{\tin})-\int_{\overline U^{\rmc}} \phi\,\d(\div \FF_{\tout}) \right),
    \end{split}
  \end{equation} 
  by noting that
  \begin{equation}
    \langle \FF_{\tin} \cdot \nu,\phi \rangle_{\partial U} = \langle m,\phi \rvert_{\partial U}\rangle_{\text{\AE}(\partial U)} = \langle \FF_{\tout} \cdot \nu,\phi \rangle_{\partial \overline U^{\rmc}}.
  \end{equation} 
  Hence $\FF \in \mathcal{DM}^{\ext}(U)$ with
  \begin{equation}
    \div \FF = \div \FF_{\tin} \mres U - \div \FF_{\tout} \mres \overline{U}^{\rmc},
  \end{equation} 
  which is supported on $\Lambda$ and satisfies the estimate \eqref{eq:bounded_twosided} by \eqref{eq:extension_estimate2} applied to $\FF_{\tin}$ and $\FF_{\tout}$.
  We can also verify the traces are attained since for $\phi \in \CC^1_{\rmb}(U)$, noting that $\FF \mres U = \FF_{\tin}$ we have
  \begin{equation}
    \langle \FF \cdot \nu, \phi\rangle_{\partial U} = \langle \FF_{\tin}\cdot\nu,\phi\rangle_{\partial U} = \langle m,\phi\rvert_{\partial U}\rangle_{\text{\AE}(\partial U)},
  \end{equation} 
  so $\mathrm{N}_U(\FF) = m$ by Theorem~\ref{lem:optimal_trace}.
  Similarly $\mathrm{N}_{\overline U^{\rmc}}(\FF) = -\mathrm{N}_{\overline U^{\rmc}}(\FF_{\tout}) = -m$, as required.
\end{proof}

\subsection{Extension of divergence-measure fields}\label{sec:extension}

As as application of Theorem~\ref{thm:surjective_trace}, we obtain the following extension theorem for fields in $\mathcal{DM}^{\ext}(\Omega)$.

\begin{thm}\label{thm:extension_main}
  Let $U \subset \mathbb R^n$ be an open set such that $\overline U^{\rmc}$ is locally uniformly quasiconvex in the sense of Definition~\ref{defn:domain_lrc} and $\partial U = \partial \overline{U}^{\rmc}$, and let $\Lambda \subset \overline{U}^{\rmc}$ satisfy the conclusion of Lemma~\ref{lem:concentration_set}.
  Then there exists a (not necessarily linear) extension operator
  \begin{equation}
    \mathcal{E}_{U} \colon \mathcal{DM}^{\ext}(U) \to \mathcal{DM}^{\ext}(\mathbb R^n)
  \end{equation} 
  such that $\mathcal{E}_{U}(\FF) \mres {U} = \FF$ for all $\FF \in \mathcal{DM}^{\ext}(U)$ and we have 
  \begin{equation}\label{eq:dm_extension_estimate}
    \lVert \mathcal{E}_{U}(\FF) \rVert_{\mathcal{DM}^{\ext}(U)} \leq C \lVert \FF \rVert_{\mathcal{DM}^{\ext}(U)} \quad\mbox{for all } \FF \in \mathcal{DM}^{\ext}(U)
  \end{equation} 
  where $C = C(n,\eps_{\overline{U}^{\rmc}},\delta_{\overline{U}^{\rmc}},\sep_{\gamma}(\Lambda,\overline{U}^{\rmc}),\dist(\Lambda,\partial U))$.
  Moreover, $\mathcal{E}_{U}(\FF)$ can be chosen to be divergence-free in $\mathbb R^n \setminus (U \cup \Lambda)$.
\end{thm}

\begin{proof}
  Let $\FF \in \mathcal{DM}^{\ext}(U)$, then by Theorem~\ref{lem:optimal_trace} we have $m := \mathrm{N}_{U}(\FF) \in \text{\AE}(\partial U) = \text{\AE}(\partial \overline{U}^{\rmc})$, with $\lVert m \rVert_{\text{\AE}(\partial U)} \leq C(n) \lVert \FF \rVert_{\mathcal{DM}^{\ext}(U)}$.
  Then by Theorem~\ref{thm:surjective_trace}, we obtain $\GG := \mathrm{E}_{\overline U^{\rmc}}(m)$ satisfying $\mathrm{N}_{\overline U^{\rmc}}(\GG) = -m$ and such that $\div \GG$ in $\overline{U}^{\rmc}$ is supported on $\Lambda$.
  Therefore if we set
  \begin{equation}
    \widetilde{\FF} = \FF \mres U + \GG \mres \overline{U}^{\rmc} \in \mathcal{M}(\mathbb R^n),
  \end{equation} 
  for any $\phi \in \CC_{\rmc}^{1}(\mathbb R^n)$ we have
  \begin{equation}
    \begin{split}
      \int_{\mathbb R^n} \nabla \phi \cdot \d\widetilde{\FF} 
      &= \int_{U} \nabla \phi \cdot \d\FF + \int_{\overline{U}^{\rmc}} \nabla\phi\cdot \d\GG \\
      &= -\langle \FF \cdot \nu,\phi\rangle_{\partial U} - \int_{U} \phi \,\d(\div \FF) - \langle \GG \cdot \nu,\phi\rangle_{\partial\overline{U}^{\rmc}} - \int_{\overline{U}^{\rmc}} \phi\,\d(\div \GG).
    \end{split}
  \end{equation} 
  Then since $\langle \FF \cdot \nu,\phi\rangle_{\partial U} = - \langle \GG \cdot \nu,\phi\rangle_{\partial\overline{U}^{\rmc}} = \langle m, \phi\rvert_{\partial U} \rangle_{\text{\AE}(\partial U)}$, the trace terms cancel to give
  \begin{equation}
    \div \widetilde{\FF} = \div \FF \mres \Omega + \div \GG \mres \overline{\Omega}^{\rmc} \in \mathcal{M}(\Omega).
  \end{equation} 
  Hence $\widetilde{\FF} \in \mathcal{DM}^{\ext}(\mathbb R^n)$ is an extension of $\FF$. Since \eqref{eq:dm_extension_estimate} follows from the corresponding estimate for $\GG$, the result follows.
\end{proof}

We also obtain extension results for divergence-free fields, which follows from the fact that, in our construction, $\div \FF$ only concentrates on the set $\Lambda$ which we can specify.
Recall from \eqref{eq:divfree_meas} in the introduction that $\mathcal{M}_{\div}(U;\mathbb R^n) \subset \mathcal{DM}^{\ext}(U)$ denotes the space of divergence-free measures in $U$.

\begin{cor}\label{cor:divfree_extension1}
Let $U \subset \mathbb R^n$ be an open set such that $\overline{U}^{\rmc}$ satisfies Definition~\ref{defn:domain_lrc} and that $\partial U = \partial \overline U^{\rmc}$.
  Then there exists an open set $\widetilde U \subset \mathbb R^n$ containing $\overline U$ such that $\dist(U,\partial \widetilde U) > 0$, and a extension operator
  \begin{equation}
    \mathcal{E}_{U,\widetilde U} \colon \mathcal{M}_{\div}(U) \to \mathcal{M}_{\div}(\widetilde U)
  \end{equation} 
  such that for all $\FF \in \mathcal{M}_{\div}(U;\mathbb R^n)$, we have $\mathcal{E}_{U,\widetilde U}(\FF) \mres U = \FF$ along with the estimate
  \begin{equation}
    \lvert \mathcal{E}_{U,\widetilde U}(\FF) \rvert(\widetilde U) \leq C \lvert \FF\rvert(U).
  \end{equation} 
\end{cor}

\begin{proof}
  By applying Theorem~\ref{thm:extension_main} we obtain an extension operator $\mathcal{E}_U$, where the divergence is prescribed in a given $\Lambda \subset \overline U^{\rmc}$. Then taking $\widetilde U = \mathbb R^n \setminus \Lambda$, Lemma~\ref{lem:concentration_set} ensures that $\dist(U,\partial\widetilde U)=\dist(\Lambda,\partial U) > 0$, and if $\FF \in \mathcal{M}_{\div}(U;\mathbb R^n)$ then $\mathcal{E}_U(\FF)$ is divergence-free away from $\Lambda$, from which the result follows.
\end{proof}

\begin{cor}\label{cor:divfree_extension2}
  Let $U \subset \mathbb R^n$ be a bounded open set such that $\overline{U}^{\rmc}$ is connected and satisfies Definition~\ref{defn:domain_lrc}, and that $\partial U = \partial \overline U^{\rmc}$.
  Then there exists a global extension operator
  \begin{equation}
    \mathcal E_{U,\mathbb R^n} \colon \mathcal{M}_{\div}(U) \to \mathcal{M}_{\div}(\mathbb R^n)
  \end{equation} 
  such that $\mathcal{E}_{U,\mathbb R^n}(\FF) \mres U = \FF$ and $\lvert \mathcal{E}_{U,\mathbb R^n}\rvert(\mathbb R^n) \leq \lvert \FF\rvert(U)$ for all $\FF \in \mathcal{M}_{\div}(U;\mathbb R^n)$.
\end{cor}

\begin{proof}
  Since $\overline{U}^{\rmc}$ is connected, by the second part of Lemma~\ref{lem:concentration_set} we can choose any $e \in \overline U^{\rmc}$ and take $\Lambda = \{e\}$.
  Then, for this choice, by Theorem~\ref{thm:extension_main} we obtain an extension operator $\mathcal{E}_U$, and we claim this extension preserves divergence-free fields.

  To see this, let $\FF \in \mathcal{M}_{\div}(U;\mathbb R^n)$ and put $\widetilde{\FF} = \mathcal{E}_U(\FF)$.
  Since $\div \widetilde{\FF}$ is divergence-free both $U$ and $\mathbb R^n \setminus (U \cup \Lambda)$, it follows that $\div \widetilde{\FF}$ is concentrated on $\{e\}$, and hence equals $\lambda \delta_e$ for some $\lambda \in \mathbb R$.
  To determine this constant, observe that $\langle \widetilde{\FF} \cdot \nu, \mathbbm{1}_{\mathbb R^n} \rangle_{\partial U} = \div(U) = 0$, and that $\lvert \widetilde{\FF}\rvert(\partial U) = 0$ by construction.
  Then by \cite[Cor.\,2.11,Rmk.\,2.12]{ChenEtAl2024} applied to $\chi \in \CC^1_{\rmc}(\mathbb R^n)$ such that $\chi \equiv 1$ in a neighbourhood of $\partial U$, it holds that
  \begin{equation}
    \langle \widetilde{\FF}\cdot\nu,\mathbbm{1}_{\mathbb R^n} \rangle_{\partial \overline U^{\rmc}} = \langle \widetilde{\FF}\cdot\nu,\chi \rangle_{\partial \overline U^{\rmc}} = -\langle \widetilde{\FF}\cdot\nu,\chi \rangle_{\partial U}= -\langle \widetilde{\FF}\cdot\nu,\mathbbm{1}_{\mathbb R^n}\rangle_{\partial U} = 0.
  \end{equation} 
  Hence
  \begin{equation}
    \lambda = \int_{\overline U^{\rmc}} \d(\div \widetilde{\FF}) = - \langle \widetilde{\FF}\cdot\nu,\mathbbm{1}_{\mathbb R^n} \rangle_{\partial \overline U^{\rmc}} = 0,
  \end{equation} 
  and so $\div \widetilde{\FF} = 0$ in $\mathbb R^n$, as required.
\end{proof}

\begin{rem}
  We note that an extension theorem for divergence-free fields in $\LL^1$ was recently established by \textsc{Gmeineder \& Schiffer} in \cite{GmeinederSchiffer2024}, based on an entirely different approach.
  While their results are restricted to bounded Lipschitz domains, in contrast to our extension results the extension operator they construct is linear and preserves $\LL^p$-regularity in the full range $1 \leq p \leq \infty$.
\end{rem}

As the following example shows, the topological condition $\partial U = \partial \overline U^{\rmc}$ imposed in Theorem~\ref{thm:extension_main} is in general necessary.

\begin{eg}\label{eg:extension_counterexample}
  We will construct a domain $U$ and a divergence-free $\LL^1$-field which admits no $\mathcal{DM}^{\ext}$-extension to any neighbourhood of $\overline U$.
  Given any $a,b \in \mathbb R^2$ such that $a \neq b$, we define $\FF_{a,b} \in \mathcal{DM}^1_{\loc}(\mathbb R^2)$ by setting
  \begin{equation}
    \FF_{a,b}(x) = \frac{x-b}{\lvert x - b\rvert^2} - \frac{x-a}{\lvert x- a\rvert^2},
  \end{equation} 
  which satisfies $\div \FF_{a,b} = 2\pi (\delta_b - \delta_a)$, along with the estimate
  \begin{equation}
    \int_{B_R(a)} \lvert \FF_{a,b}\rvert \,\d x \leq C \lvert b - a\rvert \log\left(1+ \frac{R}{\lvert b-a\rvert} \right)
  \end{equation} 
  for all $R > \lvert b - a\rvert$.
  We use this with $a = 0$ and $b = 2^{-k}\mathrm{e}_1$, setting $\FF_k = \FF_{0,2^{-k}\mathrm{e}_1}$. Then we can bound
  \begin{equation}
    \sum_{k=1}^{\infty} \int_{B_R(0)} \lvert \FF_k\rvert \,\d x \leq C \sum_{k=1}^{\infty} 2^{-k} \log\left( 1+2^k R \right)  < \infty,
  \end{equation} 
  so $\FF := \sum_{k=1}^{\infty} \FF_k$ is a well-defined vector field in $\LL^1_{\loc}(\mathbb R^2;\mathbb R^2)$.
  Furthermore setting $\Lambda = \{0\} \cup \{ 2^{-k} : k \in \mathbb N\}$ and $U = B_1(0) \setminus \Lambda$,
  we have $\div \FF = 0$ on $U$ and hence that $\FF \in \mathcal{DM}^{1}(U)$. 
  However, we claim that $\FF$ \emph{does not} admit a $\mathcal{DM}^{\ext}$-extension to $\overline U = \overline B_1(0)$. 
  Indeed suppose an extension $\widetilde{\FF} \in \mathcal{DM}^{\ext}(B_1(0))$ exists, then since $\widetilde{\FF} \mres U = \FF\,\mathcal L^2$, there exists a vector measure $\mu$ supported on $\Lambda$ such that $\widetilde{\FF} = \FF \mathcal L^2 \mres U + \mu$.
  By \cite[\S 8]{Silhavy2008} we know that $\mu \ll \mathcal H^1$ however, so $\mu = 0$ necessarily.
  On the other hand, for any $\phi \in \CC_{\rmc}^1(B_1(0))$ we have
  \begin{equation}
    \int_{B_1(0)} \nabla\phi \cdot \widetilde\FF \,\d x =  \sum_{k=1}^{\infty} \int_{B_R(0)} \nabla\phi\cdot \FF_k \,\d x = -2\pi \sum_{k=1}^{\infty} \left( \phi(2^{-k}\mathrm{e_1}) - \phi(0) \right),
  \end{equation} 
  and so
  \begin{equation}
    \div \widetilde{\FF} = 2\pi \sum_{k=1}^{\infty} (\delta_{2^{-k}\mathrm{e}_1} - \delta_0 ) \quad\text{as distributions in $B_1(0)$}.
  \end{equation} 
  However $\div \widetilde{\FF}(\{0\})$ is ill-defined, which contradicts the existence of any such extension.
\end{eg}

\subsection{The case of \texorpdfstring{$\LL^1$}{L1} fields}\label{sec:dm1}

We can also extend some of our results to hold in $\mathcal{DM}^1$, which is based on the following consequence of Proposition \ref{prop:normaltrace_weakcont}.

\begin{prop}\label{prop:dm1_rep}
  Given an open set $U \subset \mathbb R^n$, let $\FF \in \mathcal{DM}^{\ext}(U)$. Then there exists $\GG \in \mathcal{DM}^1(U)$ with $\div \GG \in \LL^1(U)$ such that
  \begin{equation}\label{eq:dm1_trace}
    \langle \FF \cdot \nu, \,\cdot\, \rangle_{\partial U} = \langle \GG \cdot \nu, \,\cdot\,\rangle_{\partial U} \quad\mbox{in $\WW^{1,\infty}(U)$},
  \end{equation} 
and $\GG$ satisfies the estimate
  \begin{equation}
    \lVert \GG \rVert_{\mathcal{DM}^1(U)} := \lVert \GG \rVert_{\LL^1(U)} + \lvert \div \GG\rvert(U) \leq 2 \lVert \FF \rVert_{\mathcal{DM}^{\ext}(U)}.
  \end{equation} 
\end{prop}

\begin{proof}
  By Proposition~\ref{prop:normaltrace_weakcont}, we know that $\langle \FF \cdot \nu\,\cdot\,\rangle_{\partial U}$ is weakly${}^{\ast}$-continuous on $\WW^{1,\infty}(U)$. Since $\WW^{1,\infty}(U) \cong \WW^{-1,1}(U)^{\ast}$ by Lemma \ref{lem:w1infty_predual}, viewing the normal trace as an element of $\WW^{-1,1}(U)^{\ast\ast}$ as in the proof of Theorem~\ref{lem:optimal_trace}, we infer that $\langle \FF \cdot \nu, \,\cdot\,\rangle_{\partial U} \in \WW^{-1,1}(U)$.
  That is, there exists $g_0,g_1,\cdots,g_n \in \LL^1(U)$ such that writing $\GG = (g_1,\cdots,g_n)$,
  \begin{equation}\label{eq:trace_negative_representation}
    \langle \FF \cdot \nu, \phi \rangle_{\partial U} = -\int_U \phi\, g_0  + \nabla \phi \cdot \GG \,\d x \quad\mbox{for all $\phi \in \WW^{1,\infty}(U)$},
  \end{equation} 
  and that 
  \begin{equation}\label{eq:negative_norm}
   \lVert g_0 \rVert_{\LL^1(U)} +  \lVert \GG \rVert_{\LL^1(U)} \leq 2 \lVert \FF \rVert_{\mathcal{DM}^1(U)}.
  \end{equation} 
  Since \eqref{eq:trace_negative_representation} vanishes for all $\phi \in \CC^1_{\rmc}(U)$ (by Proposition~\ref{prop:improved_support}), we see that 
    \begin{equation}
      \GG \in \mathcal{DM}^1(U) \quad \mbox{with } \div \GG = g_0 \in \LL^1(U),
    \end{equation} 
    so \eqref{eq:trace_negative_representation} reads as
    \begin{equation}
      \langle \FF \cdot \nu,\phi \rangle_{\partial U} = - \int_U \phi \,\div \GG \,\d x - \int_U \nabla \phi \cdot \GG \,\d x = \langle \GG \cdot \nu,\phi \rangle_{\partial U}
    \end{equation} 
    for all $\phi \in \WW^{1,\infty}(U)$, and \eqref{eq:negative_norm} reads as $\lVert \GG \rVert_{\mathcal{DM}^{1}(U)} \leq 2\lVert \FF \rVert_{\mathcal{DM}^{\ext}(U)}$.
\end{proof}

This implies the following surjectivity result for traces of $\mathcal{DM}^1$ fields. Compared to Theorem~\ref{thm:surjective_trace} we cannot specify where $\div \GG$ concentrates, however the divergence instead lies in $\LL^1$.

\begin{thm}\label{thm:dm1_surject}
  Let $U \subset \mathbb R^n$ be an open set that is locally uniformly quasiconvex in the sense of Definition~\ref{defn:domain_lrc}.
  Then restriction of the trace operator
  \begin{equation}
    \mathrm{N}_U \colon \mathcal{DM}^1(U) \twoheadrightarrow\text{\AE}(\partial U) \quad\mbox{is surjective.}
  \end{equation} 
  More precisely, for each $m \in \text{\AE}(\partial U)$ there exists $\GG \in \mathcal{DM}^1(U)$ such that $\mathrm{N}_U(\GG)= m$ and the estimate
  \begin{equation}\label{eq:dm1_estimate}
    \lVert \GG \rVert_{\mathcal{DM}^1(U)} \leq C \lVert m \rVert_{\text{\AE}(\partial U)}
  \end{equation} 
  holds, where $C = C(n,\eps_{U},\delta_U)$. Furthermore, $\div \GG \in \LL^1(U)$.
\end{thm}

\begin{proof}
  Since $\mathcal{DM}^1(U) \subset \mathcal{DM}^{\ext}(U)$, the trace operator $\mathrm{N}_U$ is well-defined by Theorem~\ref{lem:optimal_trace}.
  Given $m \in \text{\AE}(\partial U)$, by Theorem~\ref{thm:surjective_trace} there exists $\FF \in \mathcal{DM}^{\ext}(U)$ such that $\mathrm{N}_U(\FF) = m$, with the estimate $\lVert \FF \rVert_{\mathcal{DM}^{\ext}(U)} \leq C \lVert m \rVert_{\text{\AE}(\partial U)}$. By choosing $\Lambda$ as in the proof of Lemma~\ref{lem:concentration_set}, we can ensure that $C = C(n,\eps_U,\delta_U)$ (see Remark~\ref{rem:choice_lambda}).
  Now by Proposition~\ref{prop:dm1_rep}, there exists $\GG \in \mathcal{DM}^1(U)$ such that \eqref{eq:dm1_trace} holds, which implies that $\mathrm{N}_U(\GG) = \mathrm{N}_U(\FF) = m$. Also by \eqref{eq:dm1_estimate} we can estimate $\lVert \GG \rVert_{\mathcal{DM}^1(U)} \leq 2\lVert \FF \rVert_{\mathcal{DM}^{\ext}(U)} \leq C \lVert m \rVert_{\text{\AE}(\partial U)}$, so the result follows.
\end{proof}

\begin{rem}
  This result is in contrast to the following result of \textsc{Schuricht} in \cite[Prop.\,6.5]{Schuricht2007}: for $\GG \in \mathcal{DM}^1(\Omega)$, and any $U \subset \Omega$ in a suitable class of subsets $\mathcal P_h$ (depending on $\GG$) introduced in \cite[Def.\,4.1, Prop.\,4.3]{Silhavy1991}, one has $(\GG \cdot \nu)_{\partial U} \ll \mathcal H^{n-1} + \lvert \div \GG\rvert$.
  Theorem~\ref{thm:dm1_surject} shows that this is \emph{not} a general phenomenon, and is rather special to the class $\mathcal P_h$, since the fields we construct have divergences in $\LL^1$.
\end{rem}

By applying the above result to $U$ and $\overline{U}^{\rmc}$, we also obtain the following two-sided version.
We omit the proof, which is identical to that of Theorem~\ref{cor:trace_surject_twosided}, except that we apply Theorem~\ref{thm:dm1_surject} in place of Theorem~\ref{thm:surjective_trace} to obtain $\mathcal{DM}^1$-fields $\GG_{\tin}, \GG_{\tout}$ in $U$ and $\overline U^{\rmc}$ respectively.

\begin{cor}
  Let $U \subset \mathbb R^n$ be an open set such that $U$ and $\overline{U}^{\rmc}$ satisfies Definition~\ref{defn:domain_lrc}, and that $\partial U = \partial \overline U^{\rmc}$.
  Then for each $m \in \text{\AE}(\partial U)$, there exists $\GG \in \mathcal{DM}^1(\mathbb R^n)$ with $\div \GG \in \LL^1(\mathbb R^n)$ such that
  \begin{equation}
    \mathrm{N}_U(\GG) = - \mathrm{N}_{\overline U^{\rmc}}(\GG) = m,
  \end{equation} 
  and
  \begin{equation}
    \lVert \GG \rVert_{\mathcal{DM}^1(\mathbb R^n)} \leq C \lVert m \rVert_{\text{\AE}(\partial U)}.
  \end{equation} 
\end{cor}

\begin{thm}
  Let $U \subset \mathbb R^n$ such that $\overline U^{\rmc}$ satisfies Definition~\ref{defn:domain_lrc} and $\partial U = \partial \overline U^{\rmc}$.
  Then there exists a (not necessarily linear) extension operator
  \begin{equation}
    \mathcal{E}_U \colon \mathcal{DM}^1(U) \to \mathcal{DM}^{1}(\mathbb R^n)
  \end{equation} 
  such that for all $\FF \in \mathcal{DM}^1(U)$, we have $\mathcal{E}_U(\FF) \rvert_U = \FF$ and the estimate 
  \begin{equation}
    \lVert \mathcal{E}_U(\FF) \rVert_{\mathcal{DM}^1(\mathbb R^n )} \leq C\lVert \FF \rVert_{\mathcal{DM}^1(U)}
  \end{equation} 
  holds, where $C = C(n,\eps_{\overline U^{\rmc}},\delta_{\overline U^{\rmc}})$.
\end{thm}

\begin{proof}
  The proof is analogous to that of Theorem~\ref{thm:extension_main}, so we will only sketch the modifications.
  Given $\FF \in \mathcal{DM}^1(U)$, put $m = \mathrm{N}_U(\FF) \in\text{\AE}(\partial U)=\text{\AE}_{\partial\overline U^{\rmc}}$. Then by Theorem~\ref{thm:dm1_surject} applied to $\overline U^{\rmc}$ there is $\GG \in \mathcal{DM}^1(\overline U^{\rmc})$ such that $\mathrm{N}_U(\GG) = - m$, along with the estimate $\lVert \GG \rVert_{\mathcal{DM}^1(\overline U^{\rmc})} \leq C\lVert m \rVert_{\text{\AE}(\partial U)} \leq C \lVert \FF \rVert_{\mathcal{DM}^1(U)}$. Then $\widetilde \FF = \mathbbm{1}_U \FF + \mathbbm{1}_{\overline{U}^{\rmc}}\GG$ gives the desired extension.
\end{proof}

\appendix

\section{Proof of Smirnov's theorem}\label{sec:smirnov_proof}

For completeness, we will provide a proof of Smirnov's decomposition theorem, in the form stated in \S\ref{sec:smirnov}.
While the proof follows the same strategy as in the original paper, due to several technical simplifications and since we do not require a ``complete'' decomposition, the argument we give here is considerably shorter.
We point out that an alternative proof based on polyhedral approximation was established by \textsc{Paolini \& Stepanov} in \cite{PaoliniStepanov2012,PaoliniStepanov2013}, which applies in a more general setting, however we will not follow their approach.

We will use the fact that, since $\mathscr{C}_1$ is locally compact and metrisable, we can identify $\mathcal M(\mathscr{C}_1) \cong \CC_0(\mathscr{C}_1)^{\ast}$. Here $\CC_0(\mathscr{C}_1)$ is the space of continuous functions $f \colon \mathscr{C}_1 \to\mathbb R$ vanishing at infinity in the sense that $f(\gamma_k) \to 0$ whenever $\lVert \gamma_k \rVert_{\LL^{\infty}([0,1])} \to \infty$.

\begin{proof}[Proof of Theorem~\ref{thm:decomposition_fullspace}]
  We will divide the proof into two steps, starting with:

  \emph{Step 1}.\linkdest{step1} (Solenoidal case): Assume that $\div \FF = 0$. We will first prove Theorem~\ref{thm:decomposition_fullspace} under this additional assumption.

  \noindent\emph{1.1}.\,(Approximate decomposition):
  Given a standard mollifier $\eta_{\eps}$, we will set
  \begin{equation}
    f_{\eps} = \FF \ast \eta_{\eps}, \quad \tau_{\eps} = \lvert \FF \rvert\ast \eta_{\eps} + \eps \beta, \quad \sigma_{\eps} = \frac{f_{\eps}}{\tau_{\eps}},
  \end{equation} 
  where $\beta$ is everywhere positive such that $\int_{\mathbb R^n} \beta(x) \,\d x = 1$ (for instance the unit Gaussian).
  Then for each $x \in \mathbb R^n$, consider the solutions to the initial value problem
  \begin{equation}\label{eq:flow_ivp}
    \begin{cases}
      \gamma^\prime_x(t) = \sigma_{\eps}(\gamma_x(t)) &\mbox{for $t \in \mathbb R$,} \\ \gamma_x(0) = x. &
    \end{cases}
  \end{equation} 
  By the semi-group property of autonomous ODEs, this gives a $1$-parameter family of diffeomorphisms $G_t(x) = \gamma_x(t)$.
  Now by Jacobi's formula and using \eqref{eq:flow_ivp}, we can compute
  \begin{equation}
    \begin{split}
    \frac{\d}{\d t} (\det \nabla G_t(x)) 
    &= (\det \nabla G_t(x)) \tr\left( \nabla G_t(x)^{-1} \cdot \frac{\d}{\d t} \nabla G_t(x)\right)  \\
    &= (\det \nabla G_t(x)) \tr\left( \nabla \sigma_{\eps}(G_t(x)) \right)  \\
    &= (\det \nabla G_t(x)) \frac{1}{\tau_{\eps}^2} \left({\tau_{\eps}} \div f_{\eps} - \tr(f_{\eps} \cdot \nabla \tau_{\eps})\right)(G_t(x)),
    \end{split}
  \end{equation} 
  so using the chain rule and that $f_{\eps}$ is divergence-free, it follows that
  \begin{equation}\label{eq:jacobian_evolution}
    \frac{\d}{\d t} ( \tau_{\eps}(G_t(x))\det \nabla G_t(x))  = 0.
  \end{equation} 
  Evaluating at $t =0$ and noting that $\det \nabla G_t(x) \neq 0$ for all $t$, we infer that
  \begin{equation}
    \tau_{\eps}(G_t(x))\lvert \det \nabla G_t(x))\rvert = \tau_{\eps}(G_0(x))\lvert \det \nabla G_0(x))\rvert = \tau_{\eps}(x)\quad\mbox{for all } t \in \mathbb R.
  \end{equation} 
  Therefore using the change of variables $x \mapsto G_t(x)$ for each $t \in [0,1]$ and averaging over all such $t$, for any $\Phi \in \CC_{\rmb}(\mathbb R^n)$ we have
  \begin{equation}\label{eq:key_computation}
    \begin{split}
      \langle \Phi, f_{\eps} \rangle 
      &= \int_0^1 \int_{\mathbb R^n}\Phi(x)\cdot f_{\eps}(x) \,\d x \,\d t \\
      &= \int_0^1 \int_{\mathbb R^n} \Phi(\gamma_x(t))\cdot\sigma_{\eps}(\gamma_x(t)) \tau_{\eps}(\gamma_x(t)) \lvert \det \nabla G_t(x)\rvert \,\d x\,\d t \\
      &= \int_{\mathbb R^n}  \int_0^1 \Phi(\gamma_x(t)) \cdot \gamma^\prime_x(t)  \,\d t \,\tau_{\eps}(x)\,\d x\\
      &= \int_{\mathbb R^n}  \langle \llbracket \gamma_x\rvert_{[0,1]} \rrbracket, \Phi \rangle \,\tau_{\eps}(x)\,\d x,
    \end{split}
  \end{equation} 
  thereby giving a decomposition into curves for the approximating fields $f_{\eps}$.

  \noindent\emph{1.2}. (Limiting measure):
  Define the mapping
  \begin{equation}
    R_{\eps} \colon \mathbb R^n \to \mathscr C_1 \quad x \mapsto \gamma_x\rvert_{[0,1]}.
  \end{equation} 
  Then by the continuous dependence of the ODE system \eqref{eq:flow_ivp} with respect to the initial data, the mapping $R_{\eps}$ is continuous, where we equip $\mathscr{C}_1$ with the topology of uniform convergence. 
  Therefore we can define $\nu_{\eps} = (R_{\eps})_{\#} ( \tau_{\eps} \mathcal L^n)$, which is a Borel measure on $\mathscr C_1$,
  which satisfies the uniform bound 
  \begin{equation}
    \nu_{\eps}(\mathscr C_1) = \int_{\mathbb R^n} \tau_{\eps} \,\d x \leq \lvert \FF\rvert(\mathbb R^n) + \eps \quad\mbox{for each $\eps>0$}.
  \end{equation} 
  Hence by weak${}^{\ast}$-compactness of $(\nu_{\eps})_{\eps}$ in $\mathcal M(\mathscr C_1) \cong \CC_0(\mathscr C_1)^{\ast}$, there exists a subsequence $\eps_k \searrow 0$ such that $\nu_{\eps_k} \weaksto \nu$ weakly${}^{\ast}$ to a limiting measure $\nu$ which satisfies $\lvert \nu\rvert(\mathscr C_1) \leq \lvert \FF\rvert(\mathbb R^n)$. 
 To pass to the limit in \eqref{eq:key_computation}, for $\Phi \in \CC_0(\mathbb R^n;\mathbb R^n)$ we define $\ell_{\Phi} \colon \mathscr C_1 \to \mathbb R$ by 
 \begin{equation}\label{eq:ellphi_definition}
   \ell_{\Phi}(\gamma) =   \langle \llbracket \gamma \rrbracket, \Phi \rangle = \int_0^1 \Phi(\gamma(t)) \cdot \gamma^\prime(t) \,\d t \quad\mbox{for all } \gamma \in \mathscr{C}_1.
 \end{equation} 
 We claim that $\ell_{\Phi} \in \CC_0(\mathscr C_1)$, so to establish the continuity let $(\gamma_k)_k \subset \mathscr C_1$ such that $\gamma_k \to \gamma$ uniformly on $[0,1]$.
 Then as $(\gamma_k)_k$ is bounded in $\WW^{1,\infty}((0,1);\mathbb R^n)$, we also have $\gamma^\prime_k \to \gamma^\prime$ weakly${}^{\ast}$ in $\LL^{\infty}$.
 By combining this and the uniform convergence of $\Phi(\gamma_k)$, it follows that $\ell_{\Phi}(\gamma_k) \to \ell_{\Phi}(\gamma)$ as $k \to \infty$.
 Also for each $\eps>0$, since $\Phi \in \CC_0(\mathbb R^n)$, there exists $R>0$ such that $\lvert \Phi(x)\rvert \leq \eps$ for all $x \in \mathbb R^n$ with $\lvert x\rvert \geq R$. Thus if $\gamma \in \mathscr C_1$ such that $\lVert \gamma \rVert_{\LL^{\infty}([0,1])} \geq R+1$, then $\lvert \gamma(t) \rvert \geq R$ for all $t \in [0,1]$ and hence that $\lvert \ell_{\Phi}(\gamma)\rvert \leq \eps$,
 thereby showing that $\ell_{\Phi} \in \CC_0(\mathscr{C}_1)$.

 Therefore for $\Phi \in \CC_{0}(\mathbb R^n;\mathbb R^n)$, the weak${}^{\ast}$ convergence of $\nu_{\eps}$ in $\CC_0(\mathscr C_1)^{\ast}$ gives
  \begin{equation}
    \lim_{k \to \infty} \int_{\mathscr C_1} \langle\llbracket \gamma \rrbracket,\Phi \rangle \,\d \nu_{\eps_k}(\gamma) = \lim_{k \to \infty} \langle \ell_{\Phi}, \nu_{\eps_k}\rangle =\langle \ell_{\Phi}, \nu\rangle = \int_{\mathscr C_1} \langle\llbracket \gamma \rrbracket,\Phi \rangle \,\d \nu(\gamma).
  \end{equation} 
  Note the $\nu$-integral is well-defined since $\ell_{\Phi}$ is continuous on $\mathscr C_1$.
  On the other hand, by using \eqref{eq:key_computation} and the strict convergence of mollifications, we have
  \begin{equation}
    \lim_{k \to \infty} \int_{\mathscr C_1} \langle\llbracket \gamma \rrbracket,\Phi \rangle \,\d \nu_{\eps_k}(\gamma)
    = \lim_{k \to \infty} \int_{\mathbb R^n} \Phi \cdot f_{\eps_k} \,\d x = \int_{\mathbb R^n} \Phi \cdot \d \FF.
  \end{equation} 
  Equating the above two limits we obtain the decomposition
  \begin{equation}
    \int_{\mathbb R^n} \Phi \cdot \d \FF=\int_{\mathscr C_1} \langle\llbracket \gamma \rrbracket,\Phi \rangle \,\d \nu(\gamma) \quad\mbox{for all $\Phi \in \CC_{0}(\mathbb R^n;\mathbb R^n)$}.
  \end{equation} 
  Furthermore for any such $\Phi$ we can estimate
  \begin{equation}
    \begin{split}
    \left\lvert\int_{\mathbb R^n} \Phi \cdot \d \FF \right\rvert 
    &\leq  \int_{\mathscr C_1} \lvert \langle \llbracket \gamma \rrbracket,\Phi \rangle\rvert \,\d \nu(\gamma) \\
    &\leq \lVert \Phi \rVert_{\LL^{\infty}(\mathbb R^n)} \int_{\mathscr C_1} \ell(\gamma) \,\d \nu(\gamma) \\
    &\leq \lVert \Phi \rVert_{\LL^{\infty}(\mathbb R^n)}\,\nu(\mathscr C_1) \leq \lVert \Phi \rVert_{\LL^{\infty}(\mathbb R^n)}\lvert \FF\rvert(\mathbb R^n),
    \end{split}
  \end{equation} 
  since $\ell(\gamma) \leq 1$ for all $\gamma \in \mathscr C_1$.
  Taking the supremum over all $\Phi \in \CC_0(\mathbb R^n)$ with $\lVert \Phi \rVert_{\LL^\infty(\mathbb R^n)} \leq 1$, we infer that
  \begin{equation}\label{eq:decomposition_complete_fullspace}
    \lvert \FF\rvert(\mathbb R^n) \leq \int_{\mathscr C_1} {\ell(\gamma)} \,\d \nu(\gamma) \leq  \nu(\mathscr C_1) \leq \lvert \FF\rvert(\mathbb R^n),
  \end{equation} 
  so equality holds throughout.
  That is, $\nu(\mathscr C_1) = \lvert \FF\rvert(\mathbb R^n)$ and $\ell(\gamma) = 1$ for $\nu$-a.e.\,$\gamma \in \mathscr C_1$, so in particular any such $\gamma$ satisfies $\lvert \gamma^\prime(t)\rvert=1$ for $\mathcal L^1$-a.e.\,$t \in [0,1]$.

  \emph{2.3}. (Equality as measures):
  We now establish \eqref{eq:decomposition_main}, \eqref{eq:decomposition_tv} in the general case that $\Phi$ and $\phi$ are bounded Borel functions.
  For this we let $\mathcal E$ denote the set of Borel subsets $B \subset \mathbb R^n$ such that
  \begin{enumerate}[label=(\roman*)]
    \item\label{eq:E_criteria_i} the map $\ell_B$ defined to send $\gamma \mapsto \llbracket\gamma\rrbracket(B)$ is Borel measurable on $\mathscr C_1$,
    \item\label{eq:E_criteria_ii} $\displaystyle\FF(B) = \int_{\mathscr{C_1}} \llbracket\gamma\rrbracket(B) \,\d \nu(\gamma)$.
  \end{enumerate}
  We will show that $\mathcal E$ contains all Borel sets, which will imply the assertion.

  To see this, let $B = U \subset \mathbb R^n$ be an open set, and recall we have shown in the previous step that for each $\Phi \in \CC_0(\mathbb R^n)$, the mapping $\gamma \mapsto\langle \llbracket\gamma\rrbracket,\Phi \rangle$ is continuous on $\mathscr C_1$.
Now let $(\phi_k)_k\subset \CC_{\rmc}(\mathbb R^n)$ such that $0 \leq \phi_k(x) \leq 1$ and $\phi_k(x) \to \mathbbm{1}_U(x)$ as $k \to \infty$ for all $x \in \mathbb R^n$.
  Then for any $\gamma \in \mathscr C_1$ and $1 \leq i \leq n$, by the dominated convergence theorem,
  \begin{equation}\label{eq:open_approx_gamma}
    \lim_{k \to \infty} \langle \llbracket \gamma \rrbracket, \phi_k\,\mathrm{e}_i \rangle = \lim_{k \to \infty} \int_0^1 \phi_k(\gamma(t))\,\mathrm{e}_i \cdot \gamma^\prime(t) \,\d t = \llbracket \gamma \rrbracket(U) \cdot \mathrm{e}_i.
  \end{equation} 
  Therefore $\gamma \mapsto \llbracket \gamma \rrbracket (U)$ is Borel measurable, as it is componentwise the pointwise limit of continuous functions on $\mathscr{C}_1$.
  Also since $\lvert \langle\llbracket\gamma\rrbracket, \phi_k\,\mathrm{e}_i\rangle\rvert \leq 1$ for all $k$, using the pointwise convergence from \eqref{eq:open_approx_gamma}, we can apply the dominated convergence theorem which gives 
  \begin{equation}
    \FF(U) \cdot \mathrm{e}_i = \lim_{k \to \infty} \int_{\mathbb R^n} \phi_k \,\mathrm{e}_i \cdot \d \FF = \lim_{k \to \infty} \int_{\mathscr{C_1}} \langle\llbracket \gamma \rrbracket, \phi_k \,\mathrm{e}_i\rangle \,\d \nu(\gamma) = \int_{\mathscr{C}_1} \llbracket \gamma \rrbracket(U) \cdot e_i \,\d \nu(\gamma).
  \end{equation} 
  Hence $U \in \mathcal E$, and so $\mathcal E$ contains all open subsets of $\mathbb R^n$.

  Now observe that $\mathcal E$ is closed under complements, since if $B \in \mathcal E$, then we have $\ell_{\mathbb R^n \setminus B} = \ell_{\mathbb R^n} - \ell_{B}$ is Borel measurable and that \ref{eq:E_criteria_ii} is also satisfied by the linearity of the $\nu$-integral, so $\mathbb R^n \setminus B \in \mathcal E$.
  Also if $\{B_k \}_k \subset \mathcal E$ is a countable increasing sequence, setting $B = \bigcup_k B_k$ we see that $\ell_{B_k}(\gamma) \to \ell_B(\gamma)$ pointwise as $k \to \infty$ for each $\gamma \in \mathscr C_1$, so it follows that $\ell_B$ is measurable since each $\ell_{B_k}$ is.
  Moreover since $\lvert \ell_B(\gamma)\rvert \leq 1$ for all $\gamma,$ by the dominated convergence theorem we also have
  \begin{equation}
    \FF(B) = \lim_{k \to \infty} \FF(B_k) = \lim_{k \to \infty} \int_{\mathscr C_1} \llbracket \gamma \rrbracket(B_k) \,\d\nu(\gamma) = \int_{\mathscr{C}_1} \llbracket \gamma \rrbracket(B) \,\d \nu(\gamma),
  \end{equation}  
  so we also have $B \in \mathcal E$.
  Hence by the $\pi$-$\lambda$ theorem (see for instance \cite[Rmk.\,1.9]{AmbrosioEtAl2000}), $\mathcal E$ contains all Borel subsets, from which \eqref{eq:decomposition_main} follows by a density argument.

  Finally to establish \eqref{eq:decomposition_tv}, by taking a polar decomposition (using \emph{e.g.}\,\cite[Thm.\,2.22]{AmbrosioEtAl2000}), there exists $\xi \colon \mathbb R^n \to \mathbb S^{n-1}$ Borel measurable such that $\xi\cdot \FF = \lvert \FF\rvert$ as measures. Then for any Borel measurable set $B \subset \mathbb R^n$ we obtain the upper bound 
  \begin{equation}\label{eq:tv_equality_borel}
    \lvert \FF\rvert(B) = \int_{\mathbb R^n} \mathbbm{1}_B \,\xi \cdot \d \FF = \int_{\mathscr C_1} \langle \llbracket \gamma \rrbracket, \mathbbm{1}_B\,\xi \rangle \,\d \nu(\gamma) \leq \int_{\mathscr C_1} \mu_{\llbracket \gamma \rrbracket}(B)\,\d \nu(\gamma).
  \end{equation} 
  Since the same holds for $\mathbb R^n \setminus B$, and since $\lvert \FF\rvert(\mathbb R^n) = \int_{\mathscr C_1} \mu_{\llbracket \gamma \rrbracket}(\mathbb R^n) \,\d \nu$ by \eqref{eq:decomposition_complete_fullspace}, it follows that equality must hold in \eqref{eq:tv_equality_borel}, thereby establishing the result in the solenoidal case.

  \emph{Step 2}.\linkdest{step2} (Divergence-measure case):
  Now consider the case of a general $\FF \in \mathcal{DM}^{\ext}(\mathbb R^n)$. 
  We will define a field on $\mathbb R^{n+1}$ by
  \begin{equation}\label{eq:defn_S}
    {\boldsymbol S} = \FF \times (\delta_{1} - \delta_0) + \mathrm{e}_{n+1}\, \div \FF \times (\mathcal L^1 \mres (0,1)) \in \mathcal M(\mathbb R^{n+1}),
  \end{equation} 
  which satisfies $\div {\boldsymbol S} = 0$.
  Then by Step~\hyperlink{step1}{1}, there exists a measure $\tilde\nu$ on ${\mathscr C}^{n+1}_1 = \{ \tilde\gamma \in \Lip([0,1];\mathbb R^{n+1}) : \Lip(\gamma) \leq 1\}$ for which \eqref{eq:decomposition_main}, \eqref{eq:decomposition_tv} holds for ${\bm S}$.
  We then introduce the set
  \begin{equation}
    \mathcal M = \{ \tilde\gamma \in {\mathscr C}^{n+1}_1 : \tilde\gamma((0,1)) \cap (\mathbb R^n \times \{0\}) \neq \emptyset \}.
  \end{equation} 

  \emph{2.1}. (Structure of curves): We claim that, for $\tilde\nu$-a.e.\,$\tilde\gamma = (\gamma,\gamma_{n+1}) \in \mathcal M$, there exists a unique subinterval $[a,b] \subset [0,1]$ such that $\gamma = (\gamma_1,\cdots,\gamma_n)$ is constant on $[0,a]$ and $[b,1]$, and $\gamma_{n+1} \equiv 0$ on $[a,b]$.

  Indeed let $E_+, E_- \subset \mathbb R^{n}$ be a measurable partition such that $\div \FF$ is non-negative on $E_+$ and non-positive on $E_-$ (so $E_{\pm}$ is a Hahn decomposition).
  Then applying \eqref{eq:decomposition_main} and \eqref{eq:decomposition_tv} to $\mathbbm{1}_{E_{\pm} \times (0,1)}$ and using the definition of $\boldsymbol{S}$, we have
  \begin{align}
    \boldsymbol{S}(E_{\pm} \times (0,1)) &= \int_{\mathscr C_{1}^{n+1}} \int_0^1 {\tilde\gamma}'(t) \mathbbm{1}_{E_{\pm} \times (0,1)}(\gamma(t)) \,\d t \,\d \tilde\nu(\tilde\gamma) = \pm \mathrm{e}_{n+1}  \lvert \div \FF\rvert(E_{\pm}) \\
    \lvert \boldsymbol{S}\rvert(E_{\pm} \times (0,1)) &= \int_{\mathscr C_1^{n+1}} \int_0^1  \mathbbm{1}_{E_{\pm} \times I}(\gamma(t)) \,\d t \,\d \tilde\nu(\tilde\gamma) =   \lvert \div F\rvert(E_{\pm}),
  \end{align}
  where the second equality in both lines follow by noting that $\lvert \tilde{\gamma}'(t) \rvert = 1$ holds $\mathcal L^1$-a.e.\,on $[0,1]$ for $\tilde\nu$-a.e.\,$\tilde\gamma \in {\mathscr{C}}^{n+1}_1$.
  Hence, since $\lvert \boldsymbol S(E_{\pm}\times (0,1))\rvert = \lvert\boldsymbol S\rvert(E_{\pm}\times (0,1))$, it follows that
  \begin{equation}\label{eq:tildegamma_equal}
    \int_0^1 \tilde\gamma^\prime(t) \mathbbm{1}_{E_{\pm} \times (0,1)}(\gamma(t)) \,\d t = \pm \mathrm{e}_{n+1}\int_0^1 \mathbbm{1}_{E_{\pm} \times (0,1)}(\gamma(t)) \,\d t \quad\mbox{for $\nu$-a.e.\,$\tilde\gamma \in \mathscr C_{1}^{n+1}$,}
  \end{equation} 
  and thereby for any such $\tilde\gamma$, it holds that $\tilde\gamma^\prime(t) = \pm \mathrm{e}_{n+1}$ for $\mathcal L^1$-a.e.\,$t \in [0,1]$ such that $\tilde\gamma(t) \in E_{\pm} \times (0,1)$.

  Now let $\tilde\gamma  = (\gamma,\gamma_{n+1})\in \mathcal M$ such $\tilde\gamma$ is parametrised by arclength, and that \eqref{eq:tildegamma_equal} holds, which is satisfied by $\tilde\nu$-a.e.\,such $\tilde\gamma$.
  We then set
  \begin{equation}\label{eq:new_endpoints}
    a = \inf\{ t \in [0,1] \colon \gamma_{n+1}(t) = 0 \}, \quad b = \sup\{t \in [0,1] \colon \gamma_{n+1}(t) = 0 \},
  \end{equation} 
  noting that $0 \leq a \leq b \leq 1$ since $\tilde\gamma \in \mathcal M$.
  Note that\eqref{eq:tildegamma_equal} implies that $\gamma^{\prime} \equiv 0$ on $[0,a]$ and hence that $\gamma$ is constant on the same interval (observe this is vacuous if $a = 0$). Also since $\gamma_{n+1}(a) = 0$ necessarily, we must have $\gamma_{n+1}^{\prime}(a)<0$. 
  Hence it follows that $\gamma(0) \in E_-$ and that $\tilde\gamma(t) = (\gamma(0),a-t)$ on $[0,a]$.
  Similarly for $\mathcal L^1$-a.e.\,$t \in [0,1]$ for which $t>a$ and $\gamma_{n+1}(t)>0$, it follows that $\tilde\gamma^\prime(t) = \mathrm{e}_{n+1}$.
  Hence $\gamma_{n+1}(t) = 0$ for all $t \in (a,b)$ and that $\tilde\gamma(t) = (\gamma(b),t-b)$ on $[b,1]$, thereby establishing the claim.

  \emph{2.2}. (Projection): We define the mapping
  \begin{equation}
    \mathcal P \colon \mathcal M \to \mathscr{C}_1,\quad \tilde\gamma = (\gamma,\gamma_{n+1}) \mapsto \gamma(a + (b-a) \,\cdot\,) \rvert_{[0,1]},
  \end{equation} 
  where $a,b$ are defined via \eqref{eq:new_endpoints} and depends on $\tilde\gamma$.
  Since $\ell(\mathcal P(\tilde\gamma)) \leq b-a \leq 1$, this mapping is well-defined, and if $\tilde\gamma$ is of constant speed, then the same holds of $\mathcal P(\tilde\gamma)$.
  Moreover by continuity of $\tilde\gamma \in \mathcal M$, the supremum and infimum in \eqref{eq:new_endpoints} can be taken over $t \in [0,1] \cap \mathbb Q$, from which it follows that $\gamma \mapsto a, b$ are measurable as mappings $\mathcal M \to [0,1]$.
  Thus the mapping $\mathcal P$ is Borel measurable if we equip both spaces with the topology of uniform convergence, so $\nu := \mathcal P_{\#}(\tilde\nu \mres \mathcal M)$ defines a Borel measure on $\mathscr{C}_1$. 
  We will show this gives rise to the desired decomposition.

  To see this, we first observe that $\tilde\nu(\mathscr C_1^{n+1}) = \lvert {\bm S} \rvert(\mathbb R^n)$ is finite by \eqref{eq:decomposition_complete_fullspace} from Step~\hyperlink{step1}{1}, so it follows that $\nu$ is a finite Borel measure.
  Now given $\Phi \in \mathcal B_{\rmb}(\mathbb R^n;\mathbb R^n)$, define $\tilde\Phi \in \mathcal B_{\rmb}(\mathbb R^{n+1};\mathbb R^{n+1})$ by setting
  \begin{equation}
    \tilde\Phi(x,t) = 
    \begin{cases}
      (\Phi(x), 0) &\text{if } t = 0, \\
      0 &\text{if } t \neq 0.
    \end{cases}
  \end{equation} 
  Then for $\tilde\nu$-a.e.\,$\tilde\gamma\in\mathcal M$ for which the assertion of the previous step holds, we have
  \begin{equation}\label{eq:project_equality}
    \langle \llbracket \mathcal P(\tilde\gamma)\rrbracket, \Phi \rangle = \int_a^b \gamma^\prime(t)\cdot \Phi(\gamma(t)) \,\d t = \int_0^1 \tilde\gamma^\prime(t) \cdot \tilde\Phi(\tilde\gamma(t)) \,\d t = \langle \llbracket \tilde\gamma \rrbracket, \tilde\Phi \rangle.
  \end{equation} 
  Moreover if $\tilde\gamma \in \mathscr{C}_1^{n+1} \setminus \mathcal M$, then $\langle \llbracket \tilde\gamma \rrbracket, \tilde\Phi \rangle = 0$ by definition of $\mathcal M$ and $\tilde\Phi$.
  Hence 
  \begin{equation}\label{eq:project_laststep}
    \begin{split}
      \int_{\mathbb R^n} \Phi \cdot \d\FF = \int_{\mathbb R^{n+1}} \tilde\Phi \cdot \d {\bm S} 
      &= \int_{\mathscr C_1^{n+1}} \langle\llbracket\tilde\gamma\rrbracket\rangle \,\d\tilde\nu(\tilde\gamma) \\
      &= \int_{\mathcal M} \langle\llbracket\mathcal P(\tilde\gamma)\rrbracket,\Phi \rangle \,\d \tilde\nu(\tilde\gamma) = \int_{\mathscr C_1} \langle \llbracket \gamma \rrbracket, \Phi \rangle \,\d \nu(\gamma),
    \end{split}
  \end{equation} 
  establishing \eqref{eq:decomposition_main} in the general case.
  For the total variation, observe that for any $\phi \in \mathcal B_b(\mathbb R^n)$ non-negative we can  apply \eqref{eq:project_equality} with any $\Phi \in \mathcal B_b(\mathbb R^n;\mathbb R^n)$ satisfying $\lvert\Phi\rvert\leq \phi $ to deduce that
  \begin{equation}\label{eq:tv_project_equality}
    \langle \mu_{\llbracket \mathcal P \tilde\gamma \rrbracket}, \phi \rangle = \sup_{\lvert \Phi \rvert \leq \phi} \lvert \langle \llbracket\mathcal P(\tilde\gamma)\rrbracket, \Phi \rangle \rvert \leq \langle \mu_{\llbracket \tilde\gamma\rrbracket}, \tilde\phi \rangle \quad\mbox{for $\tilde\nu$-a.e.\,$\tilde\gamma\in\mathcal M$},
  \end{equation} 
  where $\phi \in \mathcal B_b(\mathbb R^n)$ is defined by setting $\tilde\phi(x,0)=\phi(x)$ and $\tilde\phi(x,t) = 0$ for all $t \neq 0$.
  Now for any such $\tilde\gamma$, since 
  \begin{equation}
    \mu_{\llbracket\mathcal P \tilde\gamma \rrbracket}(\mathbb R^n \times \{0\}) = \ell(\tilde\gamma\lvert_{[a,b]}) = \ell(\mathcal P(\tilde \gamma)) = \mu_{\llbracket\tilde\gamma \rrbracket}(\mathbb R^n),
  \end{equation} 
  we have \eqref{eq:tv_project_equality} holds with equality, and furthermore extends to signed $\phi$ by linearity.
  Hence arguing as in \eqref{eq:project_laststep} we have
  \begin{equation}
    \int_{\mathscr C_1} \langle \mu_{\llbracket\gamma \rrbracket}, \phi \rangle \,\d\nu(\gamma)=\int_{\mathcal M} \langle \mu_{\llbracket\tilde\gamma\rrbracket},\tilde\phi\rangle,\d\tilde\nu(\tilde\gamma) = \int_{\mathbb R^{n+1}} \tilde\phi \,\d\lvert {\bm S} \rvert = \int_{\mathbb R^n} \phi \,\d \lvert \FF\rvert,
  \end{equation} 
  establishing \eqref{eq:decomposition_tv}.
\end{proof}

\addcontentsline{toc}{section}{Acknowledgements}
\phantomsection
\noindent\textbf{Acknowledgements}.
This work was carried out while the author was a postdoc at Technical University Dortmund.
The author would like to thank Gui-Qiang Chen and Monica Torres for the insightful discussions concerning extensions of divergence-measure fields, and Bogdan Rai\cb{t}\u{a} for commenting on a preliminary version of the manuscript.
\printbibliography
\end{document}